\documentclass{article}
\usepackage{amsthm}
\usepackage{systeme}
\usepackage{geometry}%
\usepackage[pagebackref]{hyperref}%
\usepackage{lmodern,mathtools, amssymb}%
\usepackage[utf8]{inputenc}%
\usepackage[T1]{fontenc}%
\usepackage{todonotes}%
\usepackage{bookmark}
\usepackage{color}%
\usepackage{dsfont}
\usepackage{pgf,tikz,pgfplots}
\pgfplotsset{compat=1.14}
\usepackage{mathrsfs}
\usetikzlibrary{arrows}
\usepackage{array}
\usepackage{enumitem}  

\newcolumntype{P}[1]{>{\centering\arraybackslash}p{#1}}

\allowdisplaybreaks

\newtheorem{theorem}{Theorem}[section]%
\newtheorem{proposition}[theorem]{Proposition}%
\newtheorem{lemma}[theorem]{Lemma}%
\newtheorem{remark}[theorem]{Remark}%

\def\half{\frac{1}{2}}

\def\1{{\mathbf 1}}

\def\l{\lambda}
\def\t{\theta}

\def\e{\textrm{\bf e}}

\def\N{{\mathbb N}}

\def\Q{{\mathbb Q}}
\def\R{{\mathbb R}}

\def\P{{\mathbb P}}
\def\E{{\mathbb E}}

\title{Strong solutions to beta-Jacobi processes}
\author{Ez\'echiel Kahn \footnote{CERMICS, Ecole des Ponts, INRIA, Marne-la-Vall\'ee, France. Email:  ezechiel.kahn@enpc.fr}}
\date{\today}

\begin{document}
\maketitle

\begin{abstract}
    The purpose of this paper is to study the existence and uniqueness of solutions to a system of Stochastic Differential Equations (SDEs). The coordinates are bounded by zero and one, and repulse each other according to a Coulombian like interaction force. We show the existence of strong and pathwise unique solutions to the system until the first multiple collision at zero or one, and give a   sufficient condition on the parameters of the SDEs for this multiple collision not to occur in finite time.
\end{abstract}



\section{Introduction}
Let $p,q\ge0,\beta>0$, $n\ge1$, and $\mathbf{B}=(B^1_t,\dots,B^n_t)_t$ be a $n$-dimensional Brownian motion. Our SDEs system of interest is the following :
\begin{align}
    \label{eq:main}
    d\lambda^i_t & = 2\sqrt{\lambda^i_t(1-\lambda^i_t)}dB^i_t+\beta\left[p-(p+q)\lambda^i_t+\sum_{j\neq i}\frac{\lambda^i_t(1-\lambda^j_t)+\lambda^j_t(1-\lambda^i_t)}{\lambda^i_t-\lambda^j_t}\right]dt \tag{J(p,q)}\\
    & = 2\sqrt{\lambda^i_t(1-\lambda^i_t)}dB^i_t+\beta\left[p-n+1-(p+q)\lambda^i_t+2\lambda^i_t\sum_{j\neq i}\frac{1-\lambda^j_t}{\lambda^i_t-\lambda^j_t}\right]dt \text{ for all }i\in\{1,\dots,n\} \label{eq:maindual}\\
    &0\leq \lambda^1_t<\dots< \lambda^n_t\le1, \text{  a.s. } dt-\text{almost everywhere}. \label{eq:order}
\end{align}

The form \eqref{eq:maindual} can be found thanks to the manipulation
\begin{align}
     \sum_{j\neq i}\frac{\lambda^i_t(1-\lambda^j_t)+\lambda^j_t(1-\lambda^i_t)}{\lambda^i_t-\lambda^j_t}&=\sum_{j\neq i}\frac{2\lambda^i_t(1-\lambda^j_t)+\lambda^j_t-\lambda^i_t}{\lambda^i_t-\lambda^j_t}\nonumber\\
     &=-(n-1)+2\lambda^i_t\sum_{j\neq i}\frac{1-\lambda^j_t}{\lambda^i_t-\lambda^j_t}.\label{eq:dual_manipilation}
\end{align}
The strict inequalities (\ref{eq:order}) allow the interaction terms in the system \eqref{eq:main} to make sense.
We will look for continuous solutions to the SDEs \eqref{eq:main}. Thus, by continuity, we have for all $t\ge0$,
$0\leq \lambda^1_t\leq\dots\leq \lambda^n_t\le1, \text{  a.s.}$ While, for $n=1$, the system reduces to the real Jacobi SDE $d\lambda^1_t = 2\sqrt{\lambda^1_t(1-\lambda^1_t)}dB^1_t+\beta\left[p-(p+q)\lambda^1_t\right]dt$ studied in \cite{doumerc2005matrices} (see Lemma \ref{lemma:doumerc} in the Appendix), the coordinates are repulsed by a Coulombian like interaction when $n\ge2$.

Our goal is to study the existence and uniqueness of strong solutions to this system of SDEs in a  general setting for the parameters $p$, $q$,and $\beta$, especially in the case $\beta<1$ and $ p\wedge q-n+1<1/\beta$ which is not covered to our knowledge by the literature.
This system is related to the famous Dyson Brownian motions which satisfy 
\begin{equation*}
    d\phi^{i}_t = \sqrt{2}dB^i_t+\beta\sum_{j\neq i}\frac{dt}{\phi^{i}_t-\phi^{j}_t}, \text{ for all }i\in\{1,\dots,n\},
\end{equation*}
 where the only type of singularity is when two particles collide.
This system admits a unique strong solution (see  \cite{MR1217451}) when $\beta\ge1$. 

The difficulty in proving the existence of solutions to \eqref{eq:main} comes from the fact that there are singularities  both when a particle touches zero or one, where the derivative of the square root diffusion coefficient explodes, and when two particles touch each other. If we define $D=\{0<\lambda^1<\lambda^2<\dots<\lambda^n<1\}$, a collision occurs when the process $\Lambda=(\lambda^1_t,\dots,\lambda^n_t)_t$ hits the boundary $\partial D$ made of the union of $\{\lambda^i=\lambda^{i+1}\}$ for $i\in\{1,\dots,n-1\}$ and $\{\lambda^1=0\}$ and $\{\lambda^n=1\}$. A multiple collision occurs when two of these sets are reached at the same time and we will speak about "collision between particles" when two particles touch each other.

The same kind of difficulties emerge  in the study of the beta-Wishart system of SDEs  (see \cite{JK}) which  writes

\begin{eqnarray*}
  d\lambda_t^{i}& =& 2\sqrt{\lambda_t^{i}}dB_t^{i} + \left( \alpha-2\gamma\lambda_t^{i}+ \beta\sum_{j\ne i}\frac{\lambda_t^{i}+\lambda_t^{j}}{\lambda_t^{i}-\lambda_t^{j}} \right)dt \text{ for all  } i\in \{1,\dots,n\}, \\
  &&0\leq \lambda^1_t<\dots< \lambda^n_t, \text{  a.s. } dt-\text{a.e}.,
  \end{eqnarray*}
where $\alpha\geq0$  ,$\gamma\in\R$ and $\beta>0$. Indeed, we can find in this system the same kind of singularities both when  a particle touches  zero and when two particles collide.
The methods developped in \cite{JK} to overcome the two types of singularities and to prove the existence and uniqueness of a strong solution to the system of SDEs are adapted to the beta-Jacobi system here.

\paragraph{}
Our results about the SDEs \eqref{eq:main} are the following. In Proposition \ref{prop:multiplecollision}, we state that $k\beta(p\wedge q-n+k)\ge2$ is a  sufficient condition for multiple collisions between $k$ particles not to occur at position zero  or one in finite time.  Our main result Theorem \ref{thm:existence} gives the existence and uniqueness of solutions to the SDEs \eqref{eq:main} when $\beta\in(0,1)$ and $p\wedge q-n+1\ge\frac{1}{\beta}-1$   (so that there is no multiple collision at the origin or at position one). We explicit in Proposition \ref{prop:stationary}   the unique stationary probability measure of the SDEs \eqref{eq:main}.

The paper is organized as follows. The remaining of the introduction is devoted to the bibliographical background of this work. In Section 2, we state our main results. We prove in Section 3 some useful properties of the solutions to the system \eqref{eq:main}, before checking Proposition \ref{prop:multiplecollision} and Theorem \ref{thm:existence}  in Section 4. Section 5 is an Appendix stating some well-known results that we use in our proofs.

\paragraph{}
The beta-Jacobi process emerges from the matrix Jacobi process introduced and studied in \cite{doumerc2005matrices} : let $n,m,p>0$ be three integers,  $\Theta$  a $m\times m$ orthogonal  Brownian motion (see \cite{levy:hal-00650653} for a construction of the orthogonal Brownian motion), and let $X$ be the matrix made of the $n$ first lines and the $p$ first columns of $\Theta$. Then, if $p\ge n+1$, $q\ge n+1$ and $p+q=m$, the process $J=XX^*$, where $*$ is the adjoint operator, is a diffusion  and a solution to the SDE
\[dJ = \sqrt{J}dB\sqrt{I_n-J}+\sqrt{I_n-J}dB^*\sqrt{J}+(pI-(p+q)J)dt,\]
where $B$ is a $n\times  n$ matrix filled with Brownian motions (see \cite[Theorem 9.2.3]{doumerc2005matrices}). The eigenvalues of $J$ then verify the system of SDEs \eqref{eq:main} with $\beta=1$ (see \cite[Theorem 9.3.1]{doumerc2005matrices}). When the orthogonal Brownian motion $\Theta$ is replaced by a unitary Brownian motion,  the eigenvalues of $J$ then verify the system of SDEs \eqref{eq:main} with $\beta=2$ (see \cite{Demni2009BetaJP}). The system \eqref{eq:main} is thereby a generalization of this matrix inherited system of SDEs for general $p,q\ge0$ and $\beta>0$.
\paragraph{}
Let us consider the change of variable  $\phi^i_t=\mathrm{arcsin}(\sqrt{\lambda^i_t})$ and set $\Phi=(\phi^1_t,\dots,\phi^n_t)_t$. We apply Itô's formula, formally after the stopping time $\inf\{s\ge0:\lambda^1_s(1-\lambda^n_s)=0\}$ since the square root function is not twice differentiable at zero, and obtain, after trigonometric computations that the reader can find in \cite{Demni2009BetaJP} :

\begin{align}
    \label{eq:phi}
    d\phi^i_t&=dB^i_t+\left\{\frac{\beta(p-q)}{2}\mathrm{cot}\phi^i_t+(\beta(q-n+1)-1)\mathrm{cot}(2\phi^i_t)+\frac{\beta}{2}\sum_{j\ne i}\left[\mathrm{cot}(\phi^i_t+\phi^j_t)+\mathrm{cot}(\phi^i_t-\phi^j_t)\right]\right\}dt\\
    &0\leq \phi^1_t<\dots< \phi^n_t\le\frac{\pi}{2}, \text{  a.s. } dt-\text{a.e.}, \nonumber
\end{align}
which can be rewritten 
\begin{align}
    \label{eq:phigrad}
    d\phi^i_t&=dB^i_t-\partial_iV(\phi^1_t,\dots,\phi^n_t)dt \text{ for all }i\in\{1,\dots,n\},
\end{align}
with
\begin{equation}
\begin{split}
    V(\phi^1,\dots,\phi^n)=&-\sum_{i=1}^n\Bigg\{\frac{\beta(p-q)}{2}\ln|\sin\phi^i|+\frac{\beta(q-n+1)-1}{2}\ln|\sin (2\phi^i_t)|\\
    &+\frac{\beta}{4}\sum_{j\neq i}\left(\ln|\sin(\phi^i+\phi^j)|+\ln|\sin(\phi^i-\phi^j)|\right)\Bigg\}.\label{eq:potential}
\end{split}
\end{equation}

Systems of interacting particles following equations of the type
\begin{equation}
\label{example}
    d\phi^i_t = b_i(\phi^i_t)dt + \sigma_i (\phi^i_t)dB^i_t-\partial_i V(\phi^1_t,\dots,\phi^n_t)dt \text{ for all }i\in\{1,\dots,n\},
\end{equation}
 where $V:\mathbb{R}^n\mapsto(-\infty,+\infty]$ is a lower semi continuous convex function such that $D=\{x\in\mathbb{R}^n:V(x)<+\infty\}$ is a  non-empty convex set and $V$ is continuously differentiable in $D$ have been studied by many authors. For instance, in the peculiar case $\sigma_i(\phi^i)=\sigma>0$, $b_i(\phi^i)=0$ and $V(\phi^1,\dots,\phi^n)=-\frac{\beta}{2}\sum_{i=1}^n\sum_{j\neq i}\ln|\phi^i-\phi^j|+\frac{\theta}{2}\sum_{i=1}^n(\phi^i)^2$ with $\beta\ge\frac{\sigma^2}{2}$ and $\theta>0$, the existence and uniqueness of a strong solution to (\ref{example}) were derived in \cite{MR1217451}.

  \paragraph{Link with the multivalued stochastic differential equations theory.}
The systems of type (\ref{example}) with $b_i$ and $\sigma_i$ Lipschitz and $(\phi^1_0,\dots,\phi^n_0)\in\bar D$  were  deeply studied by Cépa and Lépingle   for instance in \cite{MR1440140} and \cite{MR1875671}  where they apply Cépa's multivalued stochastic differential equations theory developed in  \cite{MR1459451}. This theory treats the existence and uniqueness of solutions to multivalued SDEs associated with a convex function defined on a domain of $\mathbb{R}^n$. In  \cite{MR2792586}, Lépingle applied this theory to a constrained Brownian motion between reflecting or repellent walls of  Weyl chambers. The boundary behavior of the convex function dictates the behavior of the process  on these same boundaries (hitting or not the boundary in finite time, reflection on the boundary, etc.). Our SDEs of interest rewritten in the form (\ref{eq:phigrad}) thanks to the square root change of variables can be seen this way, and we will exploit this connection in the paper.

\paragraph{Link with radial Dunkl processes}(see \cite{MR2566989} for a more complete description of the theory)\

Let us define a reduced root system $R$ by a finite set in $\mathbb{R}^n\backslash \{0\}$ spanning $\mathbb{R}^n$ 
such that
\begin{itemize}
    \item for all $\alpha\in R$, $R\cap\mathbb{R}\alpha=\{\alpha,-\alpha\}$,
    \item  for all $\alpha\in R$, $\sigma_\alpha(R)=R$,
\end{itemize}
where $\sigma_\alpha$ is the reflection with respect to the hyperplane orthogonal to $\alpha$. A simple system $\Delta$ is a basis of $\mathbb{R}^n$ which induces a total ordering in $R$ the following way : a root $\alpha\in R$ is  positive if it is a positive linear combination of elements of $\Delta$. A simple system $\Delta$ being fixed, we can thus define $R_+$ as the set of positive roots of $R$.
When $\sigma_i = 1$ and $V$ takes the form 
\begin{equation}
    V : \phi\mapsto -\sum_{\alpha\in R_+}k(\alpha)\ln(\langle\alpha,\phi\rangle), \text{ }x\in D,\nonumber
\end{equation}
where $D$ is the positive Weyl chamber defined by 
\[D=\{\phi\in \mathbb{R}^n, \langle \alpha,\phi\rangle>0\text{ }\forall\alpha\in R_+\},\]
   Demni proved in \cite[Theorem 1]{MR2566989} and \cite{Demni2009BetaJP} the existence and uniqueness of a solution to (\ref{example}) on the domain $\bar D$ when $k(\alpha)>0$ for all $\alpha\in R_+$. To do so, he applied Cépa's multivalued stochastic differential equations theory. This system corresponds to \eqref{eq:main} when we make the following modification in $V$ : the convex function $\phi\mapsto-\ln(\langle\alpha,\phi\rangle)$ should be substitued by $\phi\mapsto-\ln(\sin(\langle\alpha,\phi\rangle))$,
   and for some choice of $R_+$ and $k$. 
 Indeed, when the root system is of so-called $BC_n$-type,  it is defined by 
 \begin{align}
     R &= \{\pm e_i,\pm 2e_i, \pm e_i \pm e_j, 1\leq i<j\leq n\}, \nonumber\\
     \Delta &= \{e_{i+1}-e_{i}, 1\leq i \leq n-1 , e_1\},\nonumber\\
     R_+ & = \{e_i,2e_i, 1\leq i\leq n, e_j\pm e_i, 1\leq i<j\leq n\},\nonumber\\
     D &= \left\{\phi\in\left[0,\frac{\pi}{2}\right]^n, 0<\phi^1<\dots <\phi^n<\frac{\pi}{2}\right\}, \nonumber
 \end{align}
 which,  with the right choice of $k$   gives equation (\ref{eq:phigrad})  (see (\ref{eq:potential})). 
 The condition $k(\alpha)>0$ for all $\alpha\in R_+$ implies $p> q$ and $q-n+1>\frac{1}{\beta}$. As a corollary, Demni deduces the existence and uniqueness of a strong solution to \eqref{eq:main} under the condition $p\wedge q-n+1>\frac{1}{\beta}$.
 We seek here to obtain the existence of a solution to \eqref{eq:main} while relaxing this  inequality. Demni moreover proved in \cite{Demni2009BetaJP} that for $\beta<1$, there is collision between any two neighbour particles in finite time almost surely.

\paragraph{Link with other works.}

The reader will find in Graczyk and  Malecki  \cite{MR3076363} and \cite{graczyk2014} a treatment of equations of the form 
\begin{equation*}
\begin{split}
     d\lambda^i_t&=\sigma_i(\lambda^i_t)dB^i_t+\left(b_i(\lambda^i_t)+ \sum_{j\neq i}\frac{H_{i,j}(\lambda^i_t,\lambda^j_t)}{\lambda^i_t-\lambda^j_t}\right)dt,\text{ for all } i\in\{1,\dots,n\}\\
     &\lambda^1_t\leq\dots\leq\lambda^n_t,\quad t\geq0,
\end{split}
\end{equation*}
where the functions $\sigma_i,b_i$ and $H_{i,j}$ are assumed continuous, with $H_{i,j}$ non-negative and symmetric in the sense that $H_{i,j}(x,y)=H_{j,i}(y,x)$ for all $x,y\in\mathbb{R}$.
The system \eqref{eq:main} is recovered in the particular case when $H_{i,j}(\lambda^i,\lambda^j) = \beta(\lambda^i(1+\lambda^j)+\lambda^j(1+\lambda^i))$, $\sigma_i(\lambda^i)=2\sqrt{\lambda^i(1-\lambda^i)}$ and $b_i(\lambda^i)=\beta(p-(p+q)\lambda^i)$. According to \cite[Section 6.5]{graczyk2014}, \eqref{eq:main} admits a strong solution on the time interval $[0,+\infty)$ when $\beta\geq1$. In this regime, the authors proved that there is no collision between the particles. They also demonstrated the  pathwise uniqueness of the solutions for every $\beta>0$, as recalled in Lemma \ref{lemma:pathwise_uniqueness} below.

In all these references, $\beta$ is identified as a fundamental parameter, its position relative to $1$ governing the possibility of collisions between the particles.

\paragraph{Acknowledgement} : I thank Benjamin Jourdain and Djalil Chafaï for numerous fruitful discussions. 


\section{Results}
Let us begin by the fundamental following remark.
\begin{remark}
  \label{rmk:dual}
  If $(\lambda^1_t,\dots,\lambda^n_t)_t$ is solution to \eqref{eq:main}, then $(1-\lambda^n_t,\dots,1-\lambda^1_t)_t$ is solution to $J(q,p)$ for the Brownian motion $-\mathbf{B}$.
\end{remark}

The form (\ref{eq:maindual}) of the SDEs combined with Remark \ref{rmk:dual} hint that $\beta(p\wedge q -n+1)$ is a fundamental parameter impacting the existence of solutions. Consequently, we will study the range of values of this coefficient for which the system has a solution. For instance, if we assume $p-n+1<0$, we have:

\begin{align}
  d\lambda^1_t  &= 2\sqrt{\lambda^1_t(1-\lambda^1_t)}dB^1_t+\beta\left[p-n+1-(p+q)\lambda^1_t+2\lambda^1_t\sum_{j> 1}\frac{1-\lambda^j_t}{\lambda^1_t-\lambda^j_t}\right]dt\label{eq:lambda1}\\
  &\leq2\sqrt{\lambda^1_t(1-\lambda^1_t)}dB^1_t.\nonumber
\end{align}
Then, according to the pathwise comparison theorem of Ikeda and Watanabe (that we recall in Theorem \ref{Ikeda} below) : 
\[\lambda^1_t\leq r_t \text{ a.s.  for all }t\geq0,\]
where
\[r_t=\lambda^1_0+2\int_0^t\sqrt{r_s(1-r_s)}dB^1_s\text{ for all }t\geq0,\]
which is a real Jacobi process. By Lemma \ref{lemma:doumerc},  the stopping time $T=\inf\{t\ge0:r_t=0\}$ verifies $\P(T<\infty)=1$. On $\{T<\infty\}$, after $T$, $r$ stays at zero indefinitely. As the drift in (\ref{eq:lambda1}) is negative when $\lambda^1_t=0$ and therefore when $r_t=0$, it will stay strictly negative on a time interval of positive measure. Consequently, the system of SDEs has no global solution. If we assume $q-n+1<0$ instead of $p-n+1<0$, the same reasoning can be applied replacing $(\lambda^1_t,\dots,\lambda^n_t)$ by $(1-\lambda^n_t,\dots,1-\lambda^1_t)$ and interverting $p$ and $q$ in (\ref{eq:lambda1}). 

We thus proved the following result : 

\begin{remark}
  A necessary condition for the existence of a global solution to \eqref{eq:main} is $p\wedge q -n+1>0$.
\end{remark}
We will thus assume this condition in the remaining of the paper.

It is proved in \cite[Corollary 9]{grac2013} with condition $0\le\lambda^1_0<\dots<\lambda^n_0\le1$ a.s. that for $\beta\ge1$, the SDEs \eqref{eq:main}has a unique global strong solution and that there is no collision between the particles.

Before relaxing the condition $p>q$, Demni proved in \cite{Demni2009BetaJP} by applying Cépa's multivalued equations theory (see \cite{MR1459451}) that under the conditions $p>q$,  $p\wedge q -n+1>1/\beta$, $\beta>0$ and $0\leq\lambda^1_0\leq\dots\leq\lambda^n_0\leq1$, the system of SDEs \eqref{eq:main} admits a unique strong solution. The conditions $p>q$ and $p\wedge q -n+1>1/\beta$ ensure the convexity of the potential $V$ defined in (\ref{eq:potential}) which is needed in Cépa's multivalued equations theory. In Theorem \ref{thm:existence}, we tackle the existence and uniqueness problem on wider ranges for the parameter $p\wedge q -n+1$.

Our following result applies in the case $p\wedge q -n+1\leq1/\beta$.

\begin{theorem}
  \label{thm:existence} Let us assume $\beta<1$, $p\wedge q-n+1>0$.  Let the initial  condition  $\Lambda_0=(\lambda^1_0,\dots, \lambda^n_0)$ be independent from the Brownian motion $\mathbf{B}$ and such that $0\leq\lambda^1_0\leq\dots\leq\lambda^n_0\le1$ a.s. and $\lambda^2_0(1-\lambda^{n-1}_0)>0 $ a.s.
   
   Then, the system of SDEs \eqref{eq:main} has a unique strong solution defined on the time interval $$[0,\underset{\epsilon\rightarrow0}{\lim}\zeta_\epsilon),$$ where, for  $\epsilon>0$, 
  \begin{equation}
        \zeta_\epsilon=\inf\{t\geq0 :  \lambda^1_t\leq\epsilon \text{ and }\lambda^2_t-\lambda^1_t\leq\epsilon\}\wedge \inf\{t\geq0 :  \lambda^n_t\geq1-\epsilon \text{ and }\lambda^n_t-\lambda^{n-1}_t\leq\epsilon\}.\nonumber
    \end{equation}
  
    Moreover, 
    \begin{align}
    \text{(i) } &\text{ if }p\wedge q-n+1\geq \frac{1}{\beta}-1 \text{ then }\underset{\epsilon\rightarrow0}\lim \zeta_\epsilon=\infty \text{ a.s.} \nonumber\\
    \text{(ii)  }& \text{There is no double collision between the particles, i.e. }\\
    \mathbb{P}\Big\{\exists t\in (0,&\underset{\epsilon\rightarrow0}{\lim}\zeta_\epsilon): \lambda^i_t = \lambda^{i+1}_t \text{ and } \lambda^j_t = \lambda^{j+1}_t \text{ for some }1\leq i<j\leq n-1\Big\}=0. \nonumber
    \end{align}
  \end{theorem}

  Theorem \ref{thm:existence} is proved in Section \ref{section:proofthm}.
\begin{remark}
  This result states the existence of a unique strong solution $(\lambda^1_t,\dots,\lambda^n_t)_t$ to \eqref{eq:main} defined on the time-interval $\R_+$ for $p\wedge q-n+1\geq \frac{1}{\beta}-1$   and until $\underset{\epsilon\rightarrow0}{\lim}\zeta_\epsilon$ which might be finite if $p\wedge q-n+1< \frac{1}{\beta}-1$. When $\underset{\epsilon\rightarrow0}{\lim}\zeta_\epsilon<+\infty$, then, according to the first step of the proof of assertion $(i)$ of Proposition \ref{prop:multiplecollision} below, the solution can be continuously extended to the closed-time-interval $[0,\underset{\epsilon\rightarrow0}{\lim}\zeta_\epsilon]$. The next step would be to find how to start back from $(\lambda^1_{\underset{\epsilon\rightarrow0}{\lim}\zeta_\epsilon},\dots, \lambda^n_{\underset{\epsilon\rightarrow0}{\lim}\zeta_\epsilon})$ to define a solution on the whole interval $\R_+$.
\end{remark}
The assertion $(i)$ of Theorem \ref{thm:existence} comes from the application of the following proposition for $k=2$.
\begin{proposition}[Multiple collision at zero]
\label{prop:multiplecollision}
Let $k\in\{1,\dots,n\}$. Let the initial condition $\Lambda_0=(\lambda^1_0,\dots, \lambda^n_0)$ be  independent from the Brownian motion $\mathbf{B}$ and such that $0\leq\lambda^1_0\leq\dots\leq\lambda^n_0\leq1$ a.s.
Then,

(i) if   $\lambda^k_0>0$ a.s., \eqref{eq:main} has a local solution  $\Lambda = (\lambda^1_t,\dots,\lambda^n_t)_t$  defined up to a stopping time $\mathcal{T}$ and  $k\beta(p-n+k)\geq2$, then there is no collision of $k$ particles at zero. More precisely,
\[ \mathbb{P}(\mathcal{T}=+\infty,\exists t\geq0 : \lambda_t^{1}+\lambda_t^{2}+\dots+\lambda_t^{k}=0)+\mathbb{P}(\mathcal{T}<+\infty,\underset{t\in[0,\mathcal{T})}{\inf} \lambda_t^{1}+\lambda_t^{2}+\dots+\lambda_t^{k}=0)=0.\]

(ii) if   $\lambda^{n-k+1}_0<1$ a.s., \eqref{eq:main} has a local solution  $\Lambda = (\lambda^1_t,\dots,\lambda^n_t)_t$  defined up to a stopping time $\mathcal{T}$ and  $k\beta(q-n+k)\geq2$,
then there is no collision of $k$ particles at $1$. More precisely,
\[ \mathbb{P}(\mathcal{T}=+\infty,\exists t\geq0 : \lambda_t^{n-k+1}+\lambda_t^{n-k+2}+\dots+\lambda_t^{n}=k)+\mathbb{P}(\mathcal{T}<+\infty,\underset{t\in[0,\mathcal{T})}{\sup} \lambda_t^{n-k+1}+\lambda_t^{n-k+2}+\dots+\lambda_t^{n}=k)=0.\]

\end{proposition}

This proposition is proved in the beginning of Section \ref{section:proofthm}. On top of the existence and uniqueness result, we exhibit in the next proposition the stationary probability measure of the system of SDEs \eqref{eq:main}.

\begin{proposition}\label{prop:stationary}
    Let us assume $\beta(p\wedge q-n+1)>0$. The unique stationary probability measure of the system of SDEs \eqref{eq:main} is $\rho_{inv}$ with density with respect to the Lebesgue measure 
    \[d\rho_{inv}(\lambda^1,\dots,\lambda^n) = \frac{1}{\mathcal{Z}}\times\prod_{i=1}^n\left((\lambda^i)^{\beta\frac{p-n+1}{2}-1}(1-\lambda^i)^{\beta\frac{q-n+1}{2}-1}\prod_{j\neq i}|\lambda^j-\lambda^i|^{\beta/2}\right)\mathds{1}_{\left\{0\leq\lambda^1\leq\dots\leq\lambda^n\leq1\right\}}d\lambda^1\dots d\lambda^n,\]
    where $\mathcal{Z}$ is a normalizing constant.
    
    More precisely, if a solution to \eqref{eq:main} is such that the distribution of $\Lambda_t$ does not depend on $t$, then this distribution is $\rho_{inv}$. When $\Lambda_0$ is distributed according to $\rho_{inv}$ and is independent from the Brownian motion $\mathbf{B}$, the unique solution to \eqref{eq:main} is such that for all $t\in\mathbb{R}_+$, $\Lambda_t$ is distributed according to $\rho_{inv}$, and it is a strong solution.
    \end{proposition}

    \begin{remark}

      The density of $\rho_{inv}$ is the density of the beta-Jacobi ensembles, about which
      the reader will find in \cite{MR2641363}  a documentation. These ensembles are related to the distribution of the eigenvalues of a random matrix model involving a $p\times n$ matrix and a $q\times n$ matrix with the constraints  $n\ge p+q$, and $\beta=1,2$ or $4$ depending on the dimension of the underlying algebra ($\mathbb{R}$, $\mathbb{C}$ or $\mathbb{H}$).
      \end{remark}

\section{Properties of the solutions}
The pathwise uniqueness part of the next Lemma is proved in \cite[Theorem 5.3]{graczyk2014} but we reproduce the proof for the sake of completeness.
\begin{lemma}
  \label{lemma:pathwise_uniqueness}
  The solutions to \eqref{eq:main} are pathwise unique.

  Moreover, if  $Z=(z^1_t,\dots,z^n_t)_t$ and $\tilde Z=(\tilde z^1_t,\dots,\tilde z^n_t)_t$ are two global in time solutions to \eqref{eq:main}with the same driving Brownian motion, then for all $t\geq0$,
  
  \begin{equation}
     \sum_{i=1}^n\mathbb{E}|z^i_{t}-\tilde z^i_{t}| \leq\left(\sum_{i=1}^n\mathbb{E}|z^i_0-\tilde z^i_0|\right)\exp(-2\beta(p+q) t) .\label{eq:contraction} 
 \end{equation}
\end{lemma}
\begin{proof}[Proof of Lemma \ref{lemma:pathwise_uniqueness}]\

  Let $Z=(z^1_t,\dots,z^n_t)_t$ and $\tilde Z=(\tilde z^1_t,\dots ,\tilde z^n_t)_t$ be two solutions to \eqref{eq:main} with the same driving Brownian motion. As for all $i\in\{1,\dots,n\}$, $0<\epsilon<1 $ and $t\ge0$ :
  \begin{align*}
    \int_0^t\mathds{1}_{\{0<z^i_{s}-\tilde z^i_{s}\leq\epsilon\}}\frac{d\langle z^i-\tilde z^i\rangle_s}{z^i_{s}-\tilde z^i_{s}}ds &=4\int_0^t\mathds{1}_{\{0<z^i_{s}-\tilde z^i_{s}\leq\epsilon\}}\frac{\left(\sqrt{z^i_s(1-z^i_s)}-\sqrt{\tilde z^i_s(1-\tilde z^i_s)}\right)^2}{z^i_{s}-\tilde z^i_{s}}ds\\
    &=4\int_0^t\mathds{1}_{\{0<z^i_{s}-\tilde z^i_{s}\leq\epsilon\}}\frac{\left(\sqrt{z^i_s(1-z^i_s)}-\sqrt{\tilde z^i_s(1-\tilde z^i_s)}\right)^2\left(\sqrt{z^i_s(1-z^i_s)}+\sqrt{\tilde z^i_s(1-\tilde z^i_s)}\right)}{(z^i_{s}-\tilde z^i_{s})\left(\sqrt{z^i_s(1-z^i_s)}+\sqrt{\tilde z^i_s(1-\tilde z^i_s)}\right)}ds\\
    &=4\int_0^t\mathds{1}_{\{0<z^i_{s}-\tilde z^i_{s}\leq\epsilon\}}\left|\frac{\left(z^i_s(1-z^i_s)-\tilde z^i_s(1-\tilde z^i_s)\right)\left(\sqrt{z^i_s(1-z^i_s)}-\sqrt{\tilde z^i_s(1-\tilde z^i_s)}\right)}{(z^i_{s}-\tilde z^i_{s})\left(\sqrt{z^i_s(1-z^i_s)}+\sqrt{\tilde z^i_s(1-\tilde z^i_s)}\right)}\right|ds\\
    &=4\int_0^t\mathds{1}_{\{0<z^i_{s}-\tilde z^i_{s}\leq\epsilon\}}\left|\frac{\left(1-z^i_s-\tilde z^i_s\right)\left(\sqrt{z^i_s(1-z^i_s)}-\sqrt{\tilde z^i_s(1-\tilde z^i_s)}\right)}{\left(\sqrt{z^i_s(1-z^i_s)}+\sqrt{\tilde z^i_s(1-\tilde z^i_s)}\right)}\right|ds\\
    &\le4\int_0^t|1-z^i_s-\tilde z^i_s|ds<\infty,
  \end{align*}
  the local time of $z^i-\tilde z^i$ at $0$ is zero (\cite[Lemma 3.3 p.389]{MR1725357}).

  Applying the Tanaka formula then the integration by parts formula to compute $de^{2\beta(p+q)t}|z^i_{t}-\tilde z^i_{t}|$ we get

  \begin{align}
    de^{2\beta(p+q) t}\sum_{i=1}^n|z^i_t-\tilde z^i_t| & =2\sum_{i=1}^ne^{2\beta(p+q) t}\mathrm{sgn}(z^i_t-\tilde z^i_t)\left(\sqrt{z^i_t(1-z^i_t)}-\sqrt{\tilde z^i_t(1-\tilde z^i_t)}\right)dB^i_t\nonumber\\
    &+\beta e^{2\beta(p+q) t}\sum_{i=1}^n\text{sgn}(z^i_t-\tilde z^i_t)\sum_{j\neq i}\left(\frac{z^i_t(1-z^j_t)+z^j_t(1-z^i_t)}{z^i_t-z^j_t}-\frac{\tilde z^i_t(1-\tilde z^j_t)+\tilde z^j_t(1-\tilde z^i_t)}{\tilde z^i_t-\tilde z^j_t}\right)ds\nonumber\\
    &\leq 2\sum_{i=1}^ne^{2\beta(p+q) t}\mathrm{sgn}(z^i_t-\tilde z^i_t)\left(\sqrt{z^i_t(1-z^i_t)}-\sqrt{\tilde z^i_t(1-\tilde z^i_t)}\right)dB^i_t,\label{ippp}
  \end{align}
  where $\text{sgn}(x)=1$ if $x>0$ and $\text{sgn}(x)=-1$ if $x\leq0$. The inequality (\ref{ippp}) comes from the fact that we have for all $i<j$ :
  \begin{align}
     \Biggl[\frac{z^i_s(1-z^j_s)+z^j_s(1-z^i_s)}{z^i_s-z^j_s}&-\frac{\tilde z^i_s(1-\tilde z^j_s)+\tilde z^j_s(1-\tilde z^i_s)}{\tilde z^i_s-\tilde z^j_s}\Biggr](\text{sgn}(z^i_s-\tilde z^i_s)-\text{sgn}(z^j_s-\tilde z^j_s))\nonumber\\
     &  =  2\left[\frac{z^j_s\tilde z^i_s-z^i_s\tilde z^j_s}{(z^i_s-z^j_s)(\tilde z^i_s-\tilde z^j_s)}+\frac{-z^i_sz^j_s\tilde z^i_s+z^i_sz^j_s\tilde z^j_s+z^i_s\tilde z^j_s\tilde z^i_s-z^j_s\tilde z^j_s\tilde z^i_s}{(z^i_s-z^j_s)(\tilde z^i_s-\tilde z^j_s)}\right](\text{sgn}(z^i_s-\tilde z^i_s)-\text{sgn}(z^j_s-\tilde z^j_s))\nonumber\\
    & =  2\left[\frac{z^j_s(\tilde z^i_s-z^i_s)+z^i_s(z^j_s-\tilde z^j_s)}{(z^i_s-z^j_s)(\tilde z^i_s-\tilde z^j_s)}+\frac{z^i_s\tilde z^i_s(\tilde z^j_s- z^j_s)+z^j_s\tilde z^j_s( z^i_s- \tilde z^i_s)}{(z^i_s-z^j_s)(\tilde z^i_s-\tilde z^j_s)}\right](\text{sgn}(z^i_s-\tilde z^i_s)-\text{sgn}(z^j_s-\tilde z^j_s)) \nonumber\\
    & =  2\frac{z^j_s(1-\tilde z^j_s)(\tilde z^i_s-z^i_s)+z^i_s(1-\tilde z^i_s)(z^j_s-\tilde z^j_s)}{(z^i_s-z^j_s)(\tilde z^i_s-\tilde z^j_s)}(\text{sgn}(z^i_s-\tilde z^i_s)-\text{sgn}(z^j_s-\tilde z^j_s)) \nonumber\\
    & =  -2\frac{z^j_s(1-\tilde z^j_s)|\tilde z^i_s-z^i_s|+z^i_s(1-\tilde z^i_s)|z^j_s-\tilde z^j_s| }{(z^i_s-z^j_s)(\tilde z^i_s-\tilde z^j_s)}|\text{sgn}(z^i_s-\tilde z^i_s)-\text{sgn}(z^j_s-\tilde z^j_s)|\leq 0,\label{sgncomputation}
  \end{align}
  
  as the denominator is non-negative. Let the two solutions $Z$ and $\tilde Z$ be respectively defined  on $[0,T]$ if $T<+\infty$ and  $[0,T)$ if $T=+\infty$, and $[0,\tilde T]$ if $\tilde T<+\infty$ and  $[0,\tilde T)$ if $\tilde T=+\infty$,  where $T$ and $\tilde T$ are stopping times for a filtration $(\mathcal{F}_t)_{t\geq0}$ with respect to which $\mathbf{B}$ is a Brownian motion and $Z_0$ is $\mathcal{F}_0$-measurable.

  As for all $i\in\{1,\dots,n\}$:
  \[\E\int_0^te^{4\beta(p+q)s}\left|\sqrt{z^i_s(1-z^i_s)}-\sqrt{\tilde z^i_s(1-\tilde z^i_s)}\right|^2ds<\infty,\]
  the stochastic integrals $\int_0^te^{2\beta(p+q) s}\mathrm{sgn}(z^i_s-\tilde z^i_s)\left(\sqrt{z^i_s(1-z^i_s)}-\sqrt{\tilde z^i_s(1-\tilde z^i_s)}\right)dB^i_s$ have zero expectation so that integrating (\ref{ippp}) on $[0,t\wedge T\wedge \tilde T]$ and taking expectations, we obtain that 

  \begin{equation}\label{eq:precontraction}
    \E e^{2\beta(p+q) t\wedge T\wedge \tilde T}\sum_{i=1}^n|z^i_{t\wedge T\wedge \tilde T}-\tilde z^i_{t\wedge T\wedge \tilde T}|\leq \E\sum_{i=1}^n|z^i_0-\tilde z^i_0|.
  \end{equation}

When $Z_0=\tilde Z_0$ we deduce that for all $t\ge0$
\[ \sum_{i=1}^n\E|z^i_t-\tilde z^i_t|=0,\]
which concludes the proof of pathwise uniqueness.

Taking $T=\tilde T = +\infty$ in \eqref{eq:precontraction}, we have (\ref{eq:contraction}).

\end{proof}

\begin{proof}[Proof of Proposition \ref{prop:stationary}]
    We start by proving the uniqueness of the invariant distribution.
    As the process $(\Lambda_t)_t$ lives in $[0,1]^n$,  $\rho_{inv}$ has a finite first order moment. We can thus apply the second part of Lemma \ref{lemma:pathwise_uniqueness} for two solutions to \eqref{eq:main} starting respectively according to two invariant distributions to deduce that these two invariant distributions are equal.

    The candidate density, obtained by a $t\mapsto\sin^2\left(t\right)$ change of variables from the density proportional to $e^{-2V(\phi^1,\dots,\phi^n)}$ with potential $V$ defined in (\ref{eq:potential}) which is candidate to be stationary for the gradient diffusion (\ref{eq:phigrad}), writes 

    \begin{align}
      f_{inv}(\lambda^1,\dots,\lambda^n) 
      = & \ \frac{2^{n(\beta(q-n+1)-2)}}{\mathcal{Z}'}\times\prod_{i=1}^n\left((\lambda^i)^{\beta\frac{p-n+1}{2}-1}(1-\lambda^i)^{\beta\frac{q-n+1}{2}-1}\prod_{j\neq i}|\lambda^j-\lambda^i|^{\beta/2}\right)\mathds{1}_{\left\{0\leq\lambda^1\leq\dots\leq\lambda^n\leq1\right\}},\label{eq:finv}
  \end{align}
  where $\mathcal{Z}'= \int e^{-2V(\phi^1,\dots,\phi^n)}d\phi^1\dots d\phi^n$.
  The second factor is indeed integrable since the exponents of $\lambda^i$ and $1-\lambda^i$ are bigger than $-1$.

  Let us check that the probability measure $\rho_{inv}$ with density $f_{inv}$ with respect to the Lebesgue measure solves the Fokker-Plack equation in the sense of distributions
  \begin{equation}
    \mathcal{A}^*\rho_{inv} = 0 ,\nonumber
\end{equation}
where $\mathcal{A}$ is the infinitesimal generator associated with the dynamics \eqref{eq:main}:
\begin{align*}
  \mathcal{A} &= \sum_{i=1}^nb_i(\lambda^1,\dots,\lambda^n)\frac{\partial}{\partial \lambda^i}+2\sum_{i=1}^n\lambda^i(1-\lambda^i)\frac{\partial^2}{\partial (\lambda^i)^2},
  \end{align*}
  and where $b_i(\lambda^1,\dots,\lambda^n)=\beta\left[p-(p+q)\lambda^i_t+\sum_{j\neq i}\frac{\lambda^i_t(1-\lambda^j_t)+\lambda^j_t(1-\lambda^i_t)}{\lambda^i_t-\lambda^j_t}\right]$.

    
    For a test function $\phi$ twice continuously differentiable, since $f_{inv}$ vanishes for $\lambda^i=\lambda^{i+1}$ for $i\in\{1,\dots,n-1\}$  and since  the exponents of $\lambda^i$ and $1-\lambda^i$ for $i\in\{1,\dots,n\}$ in $f_{inv}$ are bigger than $-1$ ,  we obtain by integration by parts that for $i\in\{1,\dots,n\}$
    \begin{align}
        &\int_{0\leq \lambda^1\leq\dots\leq \lambda^n\leq1}\lambda^i(1-\lambda^i)\frac{\partial^2\phi}{\partial (\lambda^i)^2}(\lambda^1,\dots,\lambda^n)f_{inv}(\lambda^1,\dots,\lambda^n)d\lambda^1\dots d\lambda^n \nonumber\\
         & = -\int_{0\leq \lambda^1\leq\dots\leq \lambda^n\leq1}\frac{\partial\phi}{\partial \lambda^i}(\lambda^1,\dots,\lambda^n)\left((1-2\lambda^i)f_{inv}(\lambda^1,\dots,\lambda^n)+\lambda^i(1-\lambda^i)\frac{\partial f_{inv}}{\partial \lambda^i}(\lambda^1,\dots,\lambda^n)\right)d\lambda^1\dots d\lambda^n.\nonumber
    \end{align} 
    
   \paragraph{}

    Then we have 
     \begin{align*}
      \int_{0\leq \lambda^1\leq\dots\leq \lambda^n\leq1}\mathcal{A}\phi(\lambda^1,\dots,\lambda^n)d\rho_{inv}(\lambda^1,\dots,\lambda^n) &=\sum_{i=1}^n\int_{0\leq \lambda^1\leq\dots\leq \lambda^n\leq1}\Biggl[(b_i(\lambda^1,\dots,\lambda^n)
      -2(1-2\lambda^i))f_{inv}(\lambda^1,\dots,\lambda^n)\\
      &-2\lambda^i(1-\lambda^i)\frac{\partial f_{inv}}{\partial \lambda^i}(\lambda^1,\dots,\lambda^n)\Biggr]\frac{\partial \phi}{\partial \lambda^i}(\lambda^1,\dots,\lambda^n)d\lambda^1\dots d\lambda^n.
    \end{align*}
    
    As for $0<\lambda^1<\dots<\lambda^n<1$ :
    \begin{align*}
      \frac{\partial f_{inv} }{\partial\lambda^i}(\lambda^1,\dots,\lambda^n)  =\left[\left(\beta\frac{p-n+1}{2}-1\right)\frac{1}{\lambda^i}-\left(\beta\frac{q-n+1}{2}-1\right)\frac{1}{1-\lambda^i}+\beta\sum_{j\neq i}\frac{1}{\lambda^i-\lambda^j}\right]f_{inv}(\lambda^1,\dots,\lambda^n) ,\nonumber
    \end{align*}


    a calculus gives us
    \begin{align*}
        (b_i(\lambda^1,\dots,\lambda^n)&
      -2(1-2\lambda^i))f_{inv}(\lambda^1,\dots,\lambda^n)-2\lambda^i(1-\lambda^i)\frac{\partial f_{inv}}{\partial \lambda^i}(\lambda^1,\dots,\lambda^n)\\ 
      =&f_{inv}(\lambda^1,\dots,\lambda^n)\Biggl[\beta\left[p-(p+q)\lambda^i +\sum_{j\neq i}\frac{\lambda^i(1-\lambda^j)+\lambda^j(1-\lambda^i)}{\lambda^i-\lambda^j}\right]-2(1-2\lambda^i)\\
      &-(\beta(p-n+1)-2)(1-\lambda^i)+(\beta(q-n+1)-2)\lambda^i-2\beta\sum_{j\neq i}\frac{\lambda^i(1-\lambda^i)}{\lambda^i-\lambda^j}\Biggr] \\
        & = 0,
    \end{align*}
    where we used the manipulation (\ref{eq:dual_manipilation}) and $\sum_{j\neq i}\frac{\lambda^i(1-\lambda^j)}{\lambda^i-\lambda^j}=(n-1)\lambda^i+\sum_{j\neq i}\frac{\lambda^i(1-\lambda^i)}{\lambda^i-\lambda^j}$.

    We conclude that $\int_{0\le\lambda^1\le\dots\le\lambda^n\le1}\mathcal{A}\phi(\lambda^1,\dots,\lambda^n)f_{inv}(\lambda^1,\dots,\lambda^n)d\lambda^1\dots d\lambda^n=0$. To deduce the existence of a weak solution to \eqref{eq:main} whose marginals follow the law $\rho_{inv}$, we may apply \cite[Theorem 2.5]{MR3485364}, as soon as
        \[\int_{0\leq \lambda^1\leq\dots\leq \lambda^n\leq1}\sum_{i=1}^n\left\{\lambda^i(1-\lambda^i)+\left|b_i(\lambda^1,\dots,\lambda^n)\right|\right\}d\rho_{inv}(\lambda^1,\dots,\lambda^n)<+\infty.\]
  This property can be proved by remarking that in the definition of (\ref{eq:finv}) of $f_{inv}$, the factor $|\lambda^j-\lambda^i|^{\beta/2}$ makes the singularity of the denominator of the interaction term $\frac{\lambda^i(1-\lambda^j)+\lambda^j(1-\lambda^i)}{\lambda^i-\lambda^j}$ integrable. By pathwise uniqueness proved in Lemma \ref{lemma:pathwise_uniqueness}, this weak solution to \eqref{eq:main} is a strong solution.

    \end{proof}




\section{Proof of Theorem \ref{thm:existence}}\label{section:proofthm}

We start this section by the proof of Proposition \ref{prop:multiplecollision} since this result is crucial in the proof of Theorem \ref{thm:existence}.

\begin{proof}[Proof of Proposition \ref{prop:multiplecollision}]
  \
  Before proving the assertion, let us first check by backwards induction that, whatever $p,q\ge 0$ and $\beta>0$, a  solution to \eqref{eq:main} defined up to a stopping time ${\mathcal T}$ is actually continuous and solves \eqref{eq:main} on the closed time interval $[0,{\cal T}]$ on $\{{\cal T}<\infty\}$. For $k=n$ and for $0\le t<\mathcal{T}$, integrating the SDE verified by the largest particle on $[0,t]$, we have :
  \[\lambda^n_t=\lambda^n_0+\int_0^t2\sqrt{\lambda^n_s(1-\lambda^n_s)}dB^n_s+\beta pt -\beta(p+q)\int_0^t\lambda^n_sds+\beta\sum_{j\neq n}\int_0^t\frac{\lambda^n_s(1-\lambda^j_s)+\lambda^j_s(1-\lambda^n_s)}{\lambda^n_s-\lambda^j_s}ds.\]

As $\lambda^n_t\in[0,1]$ for all $0\le t\le \cal T$, the right-hand side of this last equality is bounded too, and the second and the fourth terms of the right-hand side are continuous in $t$ and bounded, and thus admit a limit in $\cal T -$. Consequently, the term $\beta\sum_{j\neq n}\int_0^t\frac{\lambda^n_s(1-\lambda^j_s)+\lambda^j_s(1-\lambda^n_s)}{\lambda^n_s-\lambda^j_s}ds$ is bounded, and as all its integrands are non-negative, all the integrals in this sum are bounded and non-decreasing, and thus admits a finite limit in $\cal T-$. 

Let $k\in\{1,\dots,n-1\}$ and let us assume that for all $l\in\{k+1,\dots,n\}$, $$\sum_{j\neq l}\int_0^\mathcal{T}\left|\frac{\lambda^l_s(1-\lambda^j_s)+\lambda^j_s(1-\lambda^l_s)}{\lambda^l_s-\lambda^j_s}\right|ds<\infty.$$
For $0\leq t< \cal T$, we have :

\begin{equation*}
  \begin{split}
    \lambda^k_t=&\lambda^k_0+\int_0^t2\sqrt{\lambda^k_s(1-\lambda^k_s)}dB^k_s+\beta pt -\beta(p+q)\int_0^t\lambda^k_sds\\
    &+\beta\sum_{j< k}\int_0^t\frac{\lambda^k_s(1-\lambda^j_s)+\lambda^j_s(1-\lambda^k_s)}{\lambda^k_s-\lambda^j_s}ds-\beta\sum_{j>k}\int_0^t\frac{\lambda^k_s(1-\lambda^j_s)-\lambda^j_s(1-\lambda^k_s)}{\lambda^j_s-\lambda^k_s}ds.
  \end{split}
\end{equation*}
By induction hypothesis, each integral in the last sum of the right-hand side of this equality is bounded, and consequently as each integrand is non-negative, each integral of this term admits a limit in $\cal T-$. Taking this into account and by the same arguments as for $(\lambda^n_t)_t$, the second and the fourth terms of the right-hand side are continuous in $t$ and bounded, and thus admit a limit in $\cal T -$. Consequently the term $\beta\sum_{j< k}\int_0^t\frac{\lambda^k_s(1-\lambda^j_s)+\lambda^j_s(1-\lambda^k_s)}{\lambda^k_s-\lambda^j_s}ds$ is bounded. As all its integrands are non-negative, all the integrals in this sum are bounded and non-decreasing, and thus admits a finite limit in $\cal T-$. 

 We moreover proved that $\sum_{j\neq k}\int_0^\mathcal{T}\left|\frac{\lambda^k_s(1-\lambda^j_s)+\lambda^j_s(1-\lambda^k_s)}{\lambda^k_s-\lambda^j_s}\right|ds<\infty,$ which ends the induction argument.  Therefore, still on $\{\mathcal{T}<\infty\}$, $(\lambda^1_t,\dots,\lambda^n_t)$ admits a limit as $t\to{\mathcal T}-$. Defining $(\lambda^1_{\cal T},\lambda^2_{\cal T},\hdots,\lambda^n_{\cal T})$ as this limit, we conclude that \eqref{eq:main} is satisfied on the closed time interval $[0,{\cal T}]$ on $\{\mathcal{T}<\infty\}$. 

Now, the idea of this proof is to show that the process $\lambda^1+\dots+\lambda^k$ hits zero only when another process defined below hits zero. We then show  that this last process is not smaller than a Cox-Ingersoll-Ross (CIR) process which never hits zero.
  
  Let $s=\frac{1}{2}$, $\epsilon_s=\frac{1}{16}$, and let us define the following functions on $\R$ :
  \begin{eqnarray*}
      E(x)&:=&\e^{-\frac{1}{x}}\mathds{1}_{(0,+\infty)}(x)\\
      F(x)&:=&E(x-s)E(s+\epsilon_s-x)\\
      I(x)&:=&1-\frac{\int_{-\infty}^xF(z)dz}{\int_{-\infty}^{+\infty}F(z)dz}.
  \end{eqnarray*}
  By construction, $E$, $F$ and $I$ are non-negative smooth functions, $F$ has as support $[s,s+\epsilon_s]$ and $I$ is non-increasing, equal to $1$ on $(-\infty,s]$ and to $0$ on $[s+\epsilon_s,+\infty)$.
  Let us moreover remark that 
  \begin{eqnarray*}
      |I'(x)| &=& \frac{F(x)}{\int_{-\infty}^{+\infty}F(z)dz} \ \leq \ \frac{\e^{-\frac{2}{\epsilon_s}}}{\int_{-\infty}^{+\infty}F(z)dz}\\
      |I''(x)| & = & \left|\frac{1}{(s-x)^2}-\frac{1}{(x-s-\epsilon_s)^2}\right|\frac{|I'(x)|}{\int_{-\infty}^{+\infty}F(z)dz}\ <\ +\infty,
  \end{eqnarray*}
  as the last function is continuous and compactly supported on $\R$. We then define 
  \begin{equation}
      f := \ x \in [0,1] \ \mapsto \ I(x)\arcsin^2(\sqrt{x})+(1-I(x))\left(\frac{1}{10}x+\frac{9}{10}\right)\in[0,1].\nonumber
  \end{equation}
  The function $f$  thus coincides with $\arcsin^2(\sqrt{x})$ on $[0,s]$ and with  $\frac{1}{10}x+\frac{9}{10}$ on $[s+\epsilon_s,1]$, is increasing and is twice continuously differentiable on the whole interval $[0,1]$.
  
  Let us define for all $i\in\{1,\dots,n\}$ :  $\psi^i= f(\lambda^i)$. Applying Itô's formula (formally after the stopping time $\inf\{t\ge0:\lambda^1_t(1-\lambda^n_t)=0\}$), we get
  
  \begin{align}
      d\psi^i_t  = &f'(\lambda^i_t)d\lambda^i_t+2 f''(\lambda^i_t)\lambda^i_t(1-\lambda^i_t)dt\nonumber\\
       =& \Biggl\{2\sqrt{\psi^i_t}dB^i_t+\Bigg\{1+\beta(p-q)\sqrt{\psi^i_t}\mathrm{cot}\sqrt{\psi^i_t}+2(\beta(q-n+1)-1)\sqrt{\psi^i_t}\mathrm{cot}\left(2\sqrt{\psi^i_t}\right)\nonumber\\
      &+\beta\sqrt{\psi^i_t}\sum_{j\ne i}\left[\mathrm{cot}\left(\sqrt{\psi^i_t}+\sqrt{\tilde\psi^j_t}\right)+\mathrm{cot}\left(\sqrt{\psi^i_t}-\sqrt{\tilde\psi^j_t}\right)\right]\Bigg\}dt\Biggr\}\mathds{1}_{\left\{\psi^i_t\leq f(s)\right\}}\nonumber\\
      &+\left\{f'(\lambda^i_t)d\lambda^i_t+2 f''(\lambda^i_t)\lambda^i_t(1-\lambda^i_t)dt\right\}\mathds{1}_{\left\{\psi^i_t> f(s)\right\}},\nonumber
  \end{align}
  where we denote $\tilde{\psi}^j_t = \arcsin^2(\sqrt{f^{-1}(\psi^j_t)})$ which is always defined as $f:[0,1]\mapsto[0,1]$ is increasing and thus injective. Let us note that $\tilde{\psi}^j_t$ coincides with $\psi^j_t$ when $\psi^j_t\leq f(s)$. 

  For all $k\in\{1,\dots,n\}$, starting with
  \begin{eqnarray}
  \sigma^{k,0} &=&0,\nonumber\\
  \sigma^{k,1} &=&\inf\left\{t \in[0,\mathcal{T}): \psi_t^k\leq f\left(\frac{s}{2}\right)\right\}\mathds{1}_{\left\{\psi^k_0 \ge f(s) \right\}}, \nonumber
  \end{eqnarray}
  let us define inductively for $j\in\mathbb{N}$
  \begin{eqnarray}
     \sigma^{k,2j+1}&=&\inf\left\{t\in(\sigma^{k,2j},\mathcal{T}) : \psi_t^k\leq f\left(\frac{s}{2}\right)\right\}, \nonumber\\
     \sigma^{k,2j+2} & = & \inf\left\{t\in(\sigma^{k,2j+1},\mathcal{T}): \psi_t^k\geq f(s)\right\} , \nonumber
  \end{eqnarray}
  with the convention $\inf\emptyset=\mathcal{T}$. We can then rewrite :

  \begin{align}
    d\psi^i_t    =& \Biggl\{2\sqrt{\psi^i_t}dB^i_t+\Bigg\{1+\beta(p-q)\sqrt{\psi^i_t}\mathrm{cot}\sqrt{\psi^i_t}+2(\beta(q-n+1)-1)\sqrt{\psi^i_t}\mathrm{cot}\left(2\sqrt{\psi^i_t}\right)\nonumber\\
    &+\beta\sqrt{\psi^i_t}\sum_{j\ne i}\left[\mathrm{cot}\left(\sqrt{\psi^i_t}+\sqrt{\tilde\psi^j_t}\right)+\mathrm{cot}\left(\sqrt{\psi^i_t}-\sqrt{\tilde\psi^j_t}\right)\right]\Bigg\}dt\Biggr\}\mathds{1}_{\left\{t\in\bigcup_{l\in\N}[\sigma^{i,2l+1},\sigma^{i,2l+2})\right\}}\nonumber\\
    &+\left\{f'(\lambda^i_t)d\lambda^i_t+2 f''(\lambda^i_t)\lambda^i_t(1-\lambda^i_t)dt\right\}\mathds{1}_{\left\{t\in\bigcup_{l\in\N}[\sigma^{i,2l},\sigma^{i,2l+1})\right\}}.\label{eq:psi-time}
\end{align}
  
  We moreover define for all $i\in\{1,\dots,n\}$ and $k\ge i$,
  
  \begin{align*}
      \t_i^k(t) &=-\Biggl\{\half \beta(p-q)\left(\mathrm{cot}\sqrt{\psi^i_t}-\frac{1}{\sqrt{\psi^i_t}}\right)+(\beta(q-n+1)-1)\left(\mathrm{cot}\left(2\sqrt{\psi^i_t}\right)-\frac{1}{2\sqrt{\psi^i_t}}\right)\\
      &+\frac{\beta}{2}\sum_{j\ne i}\Biggl[\mathrm{cot}\left(\sqrt{\psi^i_t}+\sqrt{\tilde\psi^j_t}\right)-\frac{1}{\sqrt{\psi^i_t}+\sqrt{\tilde\psi^j_t}}+\mathrm{cot}\left(\sqrt{\psi^i_t}-\sqrt{\tilde\psi^j_t}\right)-\frac{1}{\sqrt{\psi^i_t}-\sqrt{\tilde\psi^j_t}}\Biggr]\Biggr\}\mathds{1}_{\left\{t\in\bigcup_{l\in\N}[\sigma^{k,2l+1},\sigma^{k,2l+2})\right\}},
  \end{align*}
  and  $\Theta^k(t) = (\t_1^k(t),\dots,\t_k^k(t),0,\dots,0)$ for $k\in\{1,\dots,n\}$.
  
  The function $\cot$ is continuous on $(0,\pi)$, $\cot(x)-\frac{1}{x}\underset{x\rightarrow0}{\sim}-\frac{x}{3}$, and the function $x\mapsto \cot(x)-\frac{1}{x}$ is decreasing on $[0,\pi)$. Thus, as  $0\leq\psi^i_t\leq\psi^k_t\leq f(s)<\frac{\pi^2}{4}$ on $\bigcup_{l\in\N}[\sigma^{k,2l+1},\sigma^{k,2l+2})\subset \bigcup_{l\in\N}[\sigma^{i,2l+1},\sigma^{i,2l+2})$, we have  for all $t\in\bigcup_{l\in\N}[\sigma^{k,2l+1},\sigma^{k,2l+2})$ : 
  
  \begin{align}   
      |\t_i^k(t)^2| &\leq \frac{\beta^2(p-q)^2}{4}\mathrm{cot}^2\left(\sqrt{f(s)}\right)+(\beta(q-n+1)-1)^2\mathrm{cot}^2\left(2\sqrt{f(s)}\right)+\frac{\beta^2}{4}\sum_{j\ne i}\left[\mathrm{cot}^2\left(2\sqrt{f(s)}\right)+\mathrm{cot}^2\left(\sqrt{f(s)}\right)\right]\nonumber\\
      & <+\infty.\label{ineq:coll_girs}
  \end{align}

  Let us now prove the assertion  by first remarking that by construction
  \[\mathbb{P}(\exists t\geq0 : \lambda_t^{1}+\lambda_t^{2}+\dots+\lambda_t^{k}=0)=\mathbb{P}(\exists t\geq0 : \psi_t^{1}+\psi_t^{2}+\dots+\psi_t^{k}=0). \] We proceed by backward induction on $k$ noticing that $\min_{k\le l \le n}l\beta(p-n+l)= k\beta(p-n+k)$. The idea is to show that the process $\psi^1+\dots+\psi^k$ is on $t\in \bigcup_{l\in\N}[\sigma^{k,2l+1},\sigma^{k,2l+2})$ not smaller than a CIR process which never hits zero.

   For $k=n$ let us consider for all $t\geq0$ :
   \[Z^n(t) = \exp\left\{\int_0^t\Theta^n(u)d\mathbf{B}_u-\half\int_0^t||\Theta^n(u)||^2du\right\}.\]
   We have thanks to the inequality (\ref{ineq:coll_girs}) :
   \[\E\left[\exp\left\{\frac{1}{2}\int_0^{t}||\Theta^n(u)||^2du\right\}\right]<\infty \ \text{for all} \ t\geq0.\]
  
       Then, according to Novikov's criterion (see for instance \cite[Proposition 5.12 p.198]{MR1121940}), $Z^n$ is a $\P$-martingale, and $\E[Z^n(t)]=1$. Consequently, recalling that $\mathcal{F}_t=\sigma\left((\l^1_0,\dots,\l^n_0),(\mathbf{B}_s)_{s\leq t}\right)$ and defining $\Q^n$ such that
       \[\frac{d\Q^n}{d\P}_{|\mathcal{F}_t}=Z(t)\]
       and for all $i\in\{1,\dots,n\}$,
       \[\tilde{B}^i_t = B^i_t -\int_0^{t}\t_i^n(s)ds,\]
       $\tilde{\mathbf{B}}=(\tilde{B}^1_t,\dots,\tilde{B}^n_t)_t$ is a $\Q^n$-Brownian motion according to the Girsanov Theorem (see for instance \cite[Proposition 5.4 p.194]{MR1121940}).
  
       Thus, (\refeq{eq:psi-time}) can be rewritten  in terms of $\tilde{\mathbf{B}}$ as  
       \begin{align}
          d\psi^i_t  =& \Biggl\{2\sqrt{\psi^i_t}d\tilde B^i_t+\Bigg[\beta(p-n+1)+2\beta\sum_{j\ne i}\frac{\psi^i_t}{\psi^i_t-\psi^j_t}\Bigg]dt\Biggr\}\mathds{1}_{\left\{t\in\bigcup_{l\in\N}[\sigma^{n,2l+1},\sigma^{n,2l+2})\right\}}\nonumber\\
          &+\left\{f'(\lambda^i_t)d\lambda^i_t+2 f''(\lambda^i_t)\lambda^i_t(1-\lambda^i_t)dt\right\}\mathds{1}_{\left\{t\in\bigcup_{l\in\N}[\sigma^{n,2l},\sigma^{n,2l+1})\right\}} \text{ for all } i\in\{1,\dots,n\}.\nonumber
      \end{align}

   We then have for all $t\in\bigcup_{l\in\N}[\sigma^{n,2l+1},\sigma^{n,2l+2})\cap[0,\cal T)$ :
   \begin{align*}
       d\left(\psi^1_t+\dots+\psi^n_t\right) &=  2\sum_{i=1}^n\sqrt{\psi^i_t}d\tilde B^i_t+n\beta pdt\\
       & = 2\sqrt{\psi^1_t+\dots+\psi^n_t}dW^n_t +n\beta pdt,
   \end{align*}
  where $W^n$ is the Brownian motion defined by :\\
  \begin{equation*}
      \begin{split}
         dW^n_t=&\mathds{1}_{\left\{t\in[0,\mathcal{T})\cap \left(\bigcup_{l\in\N}[\sigma^{n,2l+1},\sigma^{n,2l+2})\right)\right\}}\sum_{i=1}^n\left(\mathds{1}_{\left\{\sum_{j=1}^n\psi^j_t\neq0\right\}}\frac{\sqrt{\psi^i_t}}{\sqrt{\sum_{j=1}^n\psi^j_t}}+\mathds{1}_{\left\{\sum_{j=1}^n\psi^j_t=0\right\}}\frac{1}{\sqrt{n}}\right)d\tilde B^i_t\\
         &+\left(\mathds{1}_{\{t\geq\mathcal{T}\}}+\mathds{1}_{\left\{t\in[0,\mathcal{T})\cap \left(\bigcup_{l\in\N}[\sigma^{n,2l},\sigma^{n,2l+1})\right)\right\}}\right)\frac{1}{\sqrt{n}}\sum_{i=1}^nd\tilde B^i_t.
      \end{split}
  \end{equation*}
   According to Lemma \ref{lamberton}, the process $\left(\psi^1_t+\dots+\psi^n_t\right)_t$ coincides with CIR processes on each interval of $\bigcup_{l\in\N}[\sigma^{n,2l+1},\sigma^{n,2l+2})$ and each of those CIRs are defined  globally, are positive on $\R_+$ and especially at $\cal T$  a.s. when $\cal T<\infty$.   If $n\beta p\ge2$, we can thus conclude for $k=n$. 

 Let us now assume that for some $k\in\{1,\dots,n-1\}$, $k\beta(p-n+k)\geq2$ and that the desired property holds for the sum of the $k+1$ smallest particles. Since $(k+1)\beta(p-n+k+1)>k\beta(p-n+k)\geq2$, the probability of the $k+1$ smallest particles to collide at zero is null. Consequently, the probability for $k$ smallest particles to collide at zero is the probability for exactly the $k$ smallest particles to collide at zero and not the $k+1$th which stays away from zero. We thus have :
  
     \begin{multline*}
      \mathbb{P}\left(\mathcal{T}<\infty,\underset{t\in[0,\mathcal{T})}{\inf} (\lambda_t^{1}+\lambda_t^{2}+\dots+\lambda_t^{k+1})=0\right)+\mathbb{P}\left(\mathcal{T}=\infty,\exists t\geq0 : \lambda_t^{1}+\lambda_t^{2}+\dots+\lambda_t^{k+1}=0\right)\\
      =\mathbb{P}\left(\mathcal{T}<\infty,\underset{t\in[0,\mathcal{T})}{\inf} (\psi_t^{1}+\psi_t^{2}+\dots+\psi_t^{k+1})=0\right)+\mathbb{P}\left(\mathcal{T}=\infty,\exists t\geq0 : \psi_t^{1}+\psi_t^{2}+\dots+\psi_t^{k+1}=0\right)=0.
   \end{multline*}
    Then, 
    \begin{eqnarray}
      \mathbb{P}\left(\mathcal{T}=\infty,\exists t\geq0 : \psi_t^{1}+\psi_t^{2}+\dots+\psi_t^{k}=0\right)&=&\underset{\epsilon\downarrow0}\lim\mathbb{P}\Bigg(\mathcal{T}=\infty,\exists t\geq0 : \psi_t^{1}+\psi_t^{2}+\dots+\psi_t^{k} = 0 \text{ and }\psi_t^{k+1}-\psi_t^k\geq\epsilon\Bigg)  \nonumber\\
      &=&\underset{\epsilon\downarrow0}\lim\mathbb{P}\Bigg(\mathcal{T}=\infty,\exists t\geq0 : \psi_t^{1}+\psi_t^{2}+\dots+\psi_t^{k}+\mathds{1}_{\{\psi_t^{k+1}-\psi_t^k<\epsilon\}}=0\Bigg),  \nonumber
  \end{eqnarray}
  and
   \begin{eqnarray}
       \mathbb{P}\Big(\mathcal{T}<\infty,\underset{t\in[0,\mathcal{T})}{\inf}(\psi^{1}_t+\dots+\psi^{k}_t)=0\Big)
       = \underset{\epsilon\downarrow0}\lim\mathbb{P}\Bigg(\mathcal{T}<\infty,\underset{t\in[0,\mathcal{T})}{\inf} (\psi_t^{1}+\dots+\psi_t^{k}+\mathds{1}_{\{\psi_t^{k+1}-\psi_t^k<\epsilon\}})=0\Bigg).  \nonumber
   \end{eqnarray}
   
   For $\epsilon>0$, starting with
   \begin{eqnarray}
   \tau^0_\epsilon &=&\inf\{t \in[0,\mathcal{T}): \psi_t^{k+1}-\psi_t^k\geq\epsilon\}, \nonumber
   \end{eqnarray}
   let us define inductively for $j\in\mathbb{N}$
   \begin{eqnarray}
   \tau_\epsilon^{2j+1} & = & \inf\{t\in(\tau_\epsilon^{2j},\mathcal{T}) : \psi_t^{k+1}-\psi_t^k\leq\frac{\epsilon}{2}\} , \nonumber\\
   \tau_\epsilon^{2j+2}&=&\inf\{t\in(\tau_\epsilon^{2j+1},\mathcal{T}): \psi_t^{k+1}-\psi_t^k\geq\epsilon\}, \nonumber
   \end{eqnarray}
   with the convention $\inf\emptyset=\mathcal{T}$. As the function $t\rightarrow\psi_t^{k+1}-\psi_t^k$ is continuous on $[0,\mathcal{T}]$ when $\cal T<\infty$ and on $[0,\mathcal{T})$ when $\cal T=\infty$ , 
   \[\tau_\epsilon^j\underset{j\rightarrow +\infty}{\longrightarrow}\mathcal{T} \text{ and } \max\{j\in\N:\tau^j_\epsilon<\cal T\}<+\infty \text{ a.s. on }\{\cal T<\infty\}.\]

   As  $\psi_t^{k+1}-\psi_t^{k}<\epsilon$ on $[0,\tau^0_\epsilon)$ and $[\tau^{2j+1}_\epsilon,\tau^{2j+2}_\epsilon)$ for $j\in\mathbb{N}$,
   \begin{eqnarray}
       \Bigg\{\mathcal{T}<\infty &,& \underset{t\in[0,\mathcal{T})}{\inf} (\psi_t^{1}+\dots+\psi_t^{k}+\mathds{1}_{\{\psi_t^{k+1}-\psi_t^k<\epsilon\}})=0\Bigg\} =\left\{\mathcal{T}<\infty,\exists j\in\mathbb{N} ,\underset{t\in[\tau^{2j}_\epsilon,\tau^{2j+1}_\epsilon)}{\inf}(\psi_t^{1}+\dots+\psi_t^{k})=0 \right\}\nonumber\\
      \text{ and } \Bigg\{\mathcal{T}=\infty&,&\exists t\geq0 : \psi_t^{1}+\psi_t^{2}+\dots+\psi_t^{k}+\mathds{1}_{\{\psi_t^{k+1}-\psi_t^k<\epsilon\}} = 0\Bigg\} \nonumber\\ &&=\Bigg\{\mathcal{T}=\infty,\exists j\in\mathbb{N} ,\exists t\in[\tau^{2j}_\epsilon,\tau^{2j+1}_\epsilon) : \psi_t^{1}+\psi_t^{2}+\dots+\psi_t^{k} = 0 \Bigg\},\nonumber
   \end{eqnarray}

   it is enough to check that
   \begin{eqnarray}
       &&\mathbb{P}\left(\left\{\mathcal{T}<+\infty, \exists j\in\mathbb{N} ,\underset{t\in[\tau^{2j}_\epsilon,\tau^{2j+1}_\epsilon)}{\inf}(\psi_t^{1}+\dots+\psi_t^{k})=0 \right\}\right)\nonumber\\
       &+& \mathbb{P}\left(\Bigg\{\mathcal{T}=+\infty, \exists j\in\mathbb{N} ,\exists t\in[\tau^{2j}_\epsilon,\tau^{2j+1}_\epsilon) : \psi_t^{1}+\psi_t^{2}+\dots+\psi_t^{k} = 0 \Bigg\}\right) = 0. \nonumber
   \end{eqnarray}
   Moreover, these events can only happen on $t\in\bigcup_{l\in\N}[\sigma^{k,2l+1},\sigma^{k,2l+2})$ so it is in fact  enough to check that
   \begin{eqnarray}
       &&\mathbb{P}\left(\left\{\mathcal{T}<+\infty, \exists j,l\in\mathbb{N} ,\underset{t\in[\tau^{2j}_\epsilon,\tau^{2j+1}_\epsilon)\cap[\sigma^{k,2l+1},\sigma^{k,2l+2})}{\inf}(\psi_t^{1}+\dots+\psi_t^{k})=0 \right\}\right)\nonumber\\
       &+& \mathbb{P}\left(\Bigg\{\mathcal{T}=+\infty, \exists j,l\in\mathbb{N} ,\exists t\in[\tau^{2j}_\epsilon,\tau^{2j+1}_\epsilon)\cap[\sigma^{k,2l+1},\sigma^{k,2l+2}) : \psi_t^{1}+\psi_t^{2}+\dots+\psi_t^{k} = 0 \Bigg\}\right) = 0. \qquad  \label{eq:kpartpsiP}
   \end{eqnarray}
   Let us consider for all $t\geq0$
   \[Z^k(t) = \exp\left\{\int_0^t\Theta^k(u)d\mathbf{B}_u-\half\int_0^t||\Theta^k(u)||^2du\right\}.\]
   We have thanks to the inequality (\ref{ineq:coll_girs}) :
   \[\E\left[\exp\left\{\frac{1}{2}\int_0^{t}||\Theta^k(u)||^2du\right\}\right]<\infty \ \text{for all} \ t\geq0.\]
  
       Then, according to Novikov's criterion (see for instance \cite[Proposition 5.12 p.198]{MR1121940}), $Z^k$ is a $\P$-martingale, and $\E[Z^k(t)]=1$. Consequently, recalling that $\mathcal{F}_t=\sigma\left((\l^1_0,\dots,\l^n_0),(\mathbf{B}_s)_{s\leq t}\right)$ and defining $\Q^k$ such that
       \[\frac{d\Q^k}{d\P}_{|\mathcal{F}_t}=Z(t)\]
       and for all $i\in\{1,\dots,k\}$,
       \[\tilde{B}^i_t = B^i_t -\int_0^{t}\t_i^k(s)ds,\]
       $\tilde{\mathbf{B}}^{(k)}=(\tilde{B}^1_t,\dots,\tilde{B}^k_t, B^{k+1}_t,\dots,B^n_t)_t$ is a $\Q^k$-Brownian motion according to the Girsanov Theorem (see for instance \cite[Proposition 5.4 p.194]{MR1121940}). The equality \eqref{eq:kpartpsiP} is still true under the probability $\mathbb{Q}^k$.

       We  define the Brownian motion $W^k$ the following way :\\
       \begin{equation*}
           \begin{split}
              dW^k_t=&\mathds{1}_{\left\{t\in[0,\mathcal{T})\cap \left(\bigcup_{l\in\N}[\sigma^{k,2l+1},\sigma^{k,2l+2})\right)\right\}}\sum_{i=1}^k\left(\mathds{1}_{\left\{\sum_{j=1}^k\psi^j_t\neq0\right\}}\frac{\sqrt{\psi^i_t}}{\sqrt{\sum_{j=1}^k\psi^j_t}}+\mathds{1}_{\left\{\sum_{j=1}^k\psi^j_t=0\right\}}\frac{1}{\sqrt{k}}\right)d\tilde B^i_t\\
              &+\left(\mathds{1}_{\{t\geq\mathcal{T}\}}+\mathds{1}_{\left\{t\in[0,\mathcal{T})\cap \left(\bigcup_{l\in\N}[\sigma^{k,2l}_s,\sigma^{k,2l+1})\right)\right\}}\right)\frac{1}{\sqrt{k}}\sum_{i=1}^kd\tilde B^i_t.
           \end{split}
       \end{equation*}

   We have for  all $t\in[\tau_\epsilon^{2j},\tau_\epsilon^{2j+1})\cap\left(\bigcup_{l\in\N}[\sigma^{k,2l+1},\sigma^{k,2l+2})\right)$, $\psi_t^{k+1}-\psi_t^{k}>\frac{\epsilon}{2}$ and
   \begin{eqnarray}
     d(\psi^{1}_t+\dots+\psi^k_t)  &=&  2\sqrt{\psi^{1}_t+\dots+\psi^k_t}dW^k_t+k\beta(p-n+k)dt+2\beta\sum_{i=1}^k\sum_{j= k+1}^n\frac{\psi^i_t}{\psi^i_t-\psi^j_t}dt\nonumber\\
     &\geq&2\sqrt{\psi^{1}_t+\dots+\psi^k_t}dW^k_t+k\beta(p-n+k)dt-\frac{4}{\epsilon}(n-k)\beta\left(\sum_{i=1}^k\psi^i_t\right)dt. \label{kbiggerthancir}
     \end{eqnarray}

     We can then define on $\{\tau^{2j}_\epsilon\vee \sigma^{k,2l+1}<\infty\}$, for each interval  $[\tau_\epsilon^{2j}, \tau_\epsilon^{2j+1})\cap[\sigma^{k,2l+1},\sigma^{k,2l+2})$ the process $ r^{j,l}$   by $r^{j,l}_0 =\left(\lambda^{1}_{\tau_\epsilon^{2j}\vee\sigma^{k,2l+1} }+\dots+\lambda^k_{\tau_\epsilon^{2j}\vee\sigma^{k,2l+1}}\right)\mathds{1}_{\{\tau_\epsilon^{2j}\vee\sigma^{k,2l+1}<\mathcal{T}\}}+\mathds{1}_{\{\tau_\epsilon^{2j}\vee\sigma^{k,2l+1}=\mathcal{T}\}} $ and  for all $t\geq0$ :
     \begin{eqnarray}
       d r^{j,l}_t &=&2\sqrt{ r^{j,l}_t}dW^k_{t+\tau_\epsilon^{2j}\vee\sigma^{k,2l+1}}+\left[-\frac{4}{\epsilon}\beta(n-k) r^{j,l}_t+k\beta(p-n+k)\right]dt \nonumber\\
        &=&2\sqrt{ r^{j,l}_t}d(W^k_{t+\tau_\epsilon^{2j}\vee\sigma^{k,2l+1}}-W^k_{\tau_\epsilon^{2j}\vee\sigma^{k,2l+1}}+W^k_{\tau_\epsilon^{2j}\vee\sigma^{k,2l+1}})+\left[-\frac{4}{\epsilon}\beta(n-k) r^{j,l}_t+k\beta(p-n+k)\right]dt \nonumber\\
        &=&2\sqrt{ r^{j,l}_t}d\tilde W^k_{t}+\left[-\frac{4}{\epsilon}\beta(n-k) r^{j,l}_t+k\beta(p-n+k)\right]dt \nonumber
     \end{eqnarray}
     where conditionally on $\{\tau^{2j}_\epsilon\vee\sigma^{k,2l+1}<\infty\}$, by strong Markov property, $(\tilde W^k_t = W^k_{t+\tau_\epsilon^{2j}\vee\sigma^{k,2l+1}}-W^k_{\tau_\epsilon^{2j}\vee\sigma^{k,2l+1}})_{t\geq0}$ is a Brownian motion independent from $\mathcal{F}_{\tau_\epsilon^{2j}\vee\sigma^{k,2l+1}}$.
   Conditionally on $\{\tau_\epsilon^{2j}\vee\sigma^{k,2l+1}<\infty\}$, the process  $ r^{j,l}$ is a CIR process defined globally in time  according to  Lemma \ref{lamberton} with $a=k\beta(p-n+k)$ and $\sigma=2$ which satisfy $a\geq\frac{\sigma^2}{2}$, and it stays positive on $\mathbb{R}_+$.
     
    This together with (\ref{kbiggerthancir}) and Theorem \ref{Ikeda} give that for all $t\in[\tau_\epsilon^{2j}, \tau_\epsilon^{2j+1})\cap[\sigma^{k,2l+1},\sigma^{k,2l+2})$,
     \begin{equation}
         \psi^{1}_t+\dots+\psi^k_t\geq r_{t-\tau_\epsilon^{2j}\vee\sigma^{k,2l+1}}^{j,l}. \nonumber
     \end{equation}
     
      We can thus conclude the proof of $(i)$.
      
      To prove $(ii)$, we just need to apply $(i)$ to the process $(1-\lambda^n_t,\dots,1-\lambda^1_t)_t$ which satisfies $(J(q,p))$ according to Remark \ref{rmk:dual}.
  \end{proof}

\begin{proof}[Proof of Theorem \ref{thm:existence}]\
    By Remark \ref{rmk:dual}, it is enough to prove Theorem \ref{thm:existence} with the additionnal assumption $p\geq q$. Indeed, if the theorem is true under this additionnal assumption, then, if $q>p$, we apply the theorem to the system of SDEs (J(q,p)) instead of (J(p,q)), which allows to build a solution to (J(p,q)) using Remark \ref{rmk:dual}, and this solution still verifies the assertion $(ii)$ of the theorem. By Lemma \ref{lemma:pathwise_uniqueness}, the we can then conclude of the existence and uniqueness of a strong solution to \eqref{eq:main}, and this solution verifies $(ii)$. We thus make the assumption $p\ge q$  in the rest of the proof.

    As explained in the introduction, the main difficulty in proving this result comes from the fact that we have to deal with  both the singularity related to the interaction between particles and the singularity related to the edges (zero and one) when two particles hit zero (or one) at the same time. For $\epsilon>0$, our method precisely consists in separating these difficulties by defining four new systems of SDEs (\ref{eq:Jhat}), (\ref{eq:Jtilde}), (\ref{eq:Jcheck}) and (\ref{eq:Jbar}) which each remove three types of singularity and coincide with \eqref{eq:main} on domains that cover $\{t\geq0: 0\leq\lambda^1_t\leq\dots\leq\lambda^n_t\leq1\text{ and } (\lambda^1_t\geq\epsilon\text{ or }\lambda^2_t-\lambda^1_t\geq\epsilon) \text{ and } (\lambda^n_t\leq1-\epsilon\text{ or }\lambda^n_t-\lambda^{n-1}_t\geq\epsilon)\}$. This allows us to build a solution to \eqref{eq:main} by piecing together solutions to (\ref{eq:Jhat}), (\ref{eq:Jtilde}), (\ref{eq:Jcheck}) and (\ref{eq:Jbar}).
  
    Let us consider in this proof the Brownian motion $\mathbf{B}=(B^1_t,\dots,B^n_t)_{t}$, $\mathcal{F}_t=\sigma\left((\lambda^1_0,\dots,\lambda^n_0), \left(\mathbf{B}_s \right)_{s \le t} \right)$ for all $t\geq0$, and the system of SDEs defined by \eqref{eq:main}.
    Let us define for all $\epsilon>0$ the following SDEs :
    
    \begin{equation}
    \label{eq:Jhat}
    \begin{split}
       \forall i\in\{1,\dots,n\}, d \hat\lambda^{i,\epsilon}_t =\ & 2\sqrt{\hat\lambda^{i,\epsilon}_t(1-\hat\lambda^{i,\epsilon}_t)}dB^i_t \\
       &+\Biggl\{(1-2\hat\lambda^{i,\epsilon}_t)\left[2-0\vee\frac{2\sqrt{2}}{\sqrt{\epsilon}}\left(\sqrt{1-\hat\lambda^{i,\epsilon}_t}-\frac{\sqrt{\epsilon}}{2\sqrt{2}}\right)\wedge1 -0\vee\frac{2\sqrt{2}}{\sqrt{\epsilon}}\left(\sqrt{\hat\lambda^{i,\epsilon}_t}-\frac{\sqrt{\epsilon}}{2\sqrt{2}}\right)\wedge1 \right]\\
        &+\beta\left[p-(p+q)\hat\lambda^{i,\epsilon}_t+\sum_{j\neq i}\frac{\hat\lambda^{i,\epsilon}_t(1-\hat\lambda^{j,\epsilon}_t)+\hat\lambda^{j,\epsilon}_t(1-\hat\lambda^{i,\epsilon}_t)}{\hat\lambda^{i,\epsilon}_t-\hat\lambda^{j,\epsilon}_t}\right]\Biggr\}dt \\
        &0\leq \hat\lambda^{1,\epsilon}_t<\dots< \hat\lambda^{n,\epsilon}_t\leq1, \text{ a.s.,  } dt-a.e.
    \end{split}\tag{$\hat{J} _\epsilon$}
    \end{equation}
  
    \begin{equation}
        \label{eq:Jtilde}
  \left\{\begin{aligned}
      d\tilde\lambda^{1,\epsilon}_t = \ & 2\sqrt{\tilde\lambda^{1,\epsilon}_t(1-\tilde\lambda^{1,\epsilon}_t)}dB^1_t+\beta\left[p-n+1-(p+q)\tilde\lambda^{1,\epsilon}_t-2\frac{(1-\tilde\lambda^{1,\epsilon}_t)\wedge\frac{\epsilon}{2}}{\frac{\epsilon}{2}}\sum_{j\neq1}\frac{\tilde\lambda^{1,\epsilon}_t(1-\tilde\lambda^{j,\epsilon}_t)}{(\tilde\lambda^{j,\epsilon}_t-\tilde\lambda^{1,\epsilon}_t\wedge\epsilon)\vee\epsilon}\right]dt \\
  \forall i\in\{2,\dots, n \}\text{, }  d \tilde\lambda^{i,\epsilon}_t =&\ 2\sqrt{\tilde\lambda^{i,\epsilon}_t(1-\tilde\lambda^{i,\epsilon}_t)}dB^i_t\\
  &+\Biggl\{(1-2\tilde\lambda^{i,\epsilon}_t)\left[2-0\vee\frac{2\sqrt{2}}{\sqrt{\epsilon}}\left(\sqrt{1-\tilde\lambda^{i,\epsilon}_t}-\frac{\sqrt{\epsilon}}{2\sqrt{2}}\right)\wedge1 -0\vee\frac{2}{\sqrt{\epsilon}}\left(\sqrt{\tilde\lambda^{i,\epsilon}_t}-\frac{\sqrt{\epsilon}}{2}\right)\wedge1 \right]\\
  &+\beta\left[p-n+1-(p+q)\tilde\lambda^{i,\epsilon}_t+2\frac{(1-\tilde\lambda^{i,\epsilon}_t)\wedge\frac{\epsilon}{2}}{\frac{\epsilon}{2}}\frac{\tilde\lambda^{i,\epsilon}_t(1-\tilde\lambda^{1,\epsilon}_t)}{(\tilde\lambda^{i,\epsilon}_t-\tilde\lambda^{1,\epsilon}_t)\vee\epsilon}+2\sum_{j\geq2,j\neq i}\frac{\tilde\lambda^{i,\epsilon}_t(1-\tilde\lambda^{j,\epsilon}_t)}{\tilde\lambda^{i,\epsilon}_t-\tilde\lambda^{j,\epsilon}_t}\right]\Biggl\}dt\\
  &0\leq \tilde\lambda^{1,\epsilon}_t\leq1 \text{ and } 0\leq \tilde\lambda^{2,\epsilon}_t<\dots< \tilde\lambda^{n,\epsilon}_t\leq1, \text{ a.s.,  } dt-a.e.
  \end{aligned}\right.
  \tag{$\tilde{J} _\epsilon$}
  \end{equation}

  \begin{equation}
    \label{eq:Jcheck}
\left\{\begin{aligned}
  d\check\lambda^{n,\epsilon}_t = \ & 2\sqrt{\check\lambda^{n,\epsilon}_t(1-\check\lambda^{n,\epsilon}_t)}dB^n_t\\
  &+\beta\left[p-n+1-(p+q)\check\lambda^{n,\epsilon}_t+2\frac{\check\lambda^{n,\epsilon}_t\wedge\frac{\epsilon}{2}}{\frac{\epsilon}{2}}\sum_{j\neq n}\frac{\check\lambda^{n,\epsilon}_t(1-\check\lambda^{j,\epsilon}_t)}{(\check\lambda^{n,\epsilon}_t-\check\lambda^{j,\epsilon}_t\wedge\epsilon)\vee\epsilon}\right]dt \\
\forall i\in\{1,\dots, n-1 \}\text{, }  d \check\lambda^{i,\epsilon}_t =&\ 2\sqrt{\check\lambda^{i,\epsilon}_t(1-\check\lambda^{i,\epsilon}_t)}dB^i_t\\
&+\Biggl\{(1-2\check\lambda^{i,\epsilon}_t)\left[2-0\vee\frac{2}{\sqrt{\epsilon}}\left(\sqrt{1-\check\lambda^{i,\epsilon}_t}-\frac{\sqrt{\epsilon}}{2}\right)\wedge1 -0\vee\frac{2\sqrt{2}}{\sqrt{\epsilon}}\left(\sqrt{\check\lambda^{i,\epsilon}_t}-\frac{\sqrt{\epsilon}}{2\sqrt{2}}\right)\wedge1 \right]\\
&+\beta\Biggl[p-n+1-(p+q)\check\lambda^{i,\epsilon}_t-2\frac{(1-\check\lambda^{i,\epsilon}_t)\wedge\frac{\epsilon}{2}}{\frac{\epsilon}{2}}\frac{\check\lambda^{i,\epsilon}_t(1-\check\lambda^{n,\epsilon}_t)}{(\check\lambda^{n,\epsilon}_t-\check\lambda^{i,\epsilon}_t)\vee\epsilon}+2\sum_{j\leq n-1,j\neq i}\frac{\check\lambda^{i,\epsilon}_t(1-\check\lambda^{j,\epsilon}_t)}{\check\lambda^{i,\epsilon}_t-\check\lambda^{j,\epsilon}_t}\Biggr]\Biggl\}dt\\
&0\leq \check\lambda^{n,\epsilon}_t\leq1 \text{ and } 0\leq \check\lambda^{1,\epsilon}_t<\dots< \check\lambda^{n-1,\epsilon}_t\leq1, \text{ a.s.,  } dt-a.e.
\end{aligned}\right.
\tag{$\check{J} _\epsilon$}
\end{equation}

\begin{equation}
    \label{eq:Jbar}
\left\{\begin{aligned}
  d\bar\lambda^{1,\epsilon}_t = \ & 2\sqrt{\bar\lambda^{1,\epsilon}_t(1-\bar\lambda^{1,\epsilon}_t)}dB^1_t+\beta\left[p-n+1-(p+q)\bar\lambda^{1,\epsilon}_t-2\frac{(1-\bar\lambda^{1,\epsilon}_t)\wedge\frac{\epsilon}{2}}{\frac{\epsilon}{2}}\sum_{j\neq 1}\frac{\bar\lambda^{1,\epsilon}_t(1-\bar\lambda^{j,\epsilon}_t)}{(\bar\lambda^{j,\epsilon}_t-\bar\lambda^{1,\epsilon}_t\wedge\epsilon)\vee\epsilon}\right]dt \\
  d\bar\lambda^{n,\epsilon}_t = \ & 2\sqrt{\bar\lambda^{n,\epsilon}_t(1-\bar\lambda^{n,\epsilon}_t)}dB^n_t+\beta\left[p-n+1-(p+q)\bar\lambda^{n,\epsilon}_t+2\frac{\bar\lambda^{n,\epsilon}_t\wedge\frac{\epsilon}{2}}{\frac{\epsilon}{2}}\sum_{j\neq n}\frac{\bar\lambda^{n,\epsilon}_t(1-\bar\lambda^{j,\epsilon}_t)}{(\bar\lambda^{n,\epsilon}_t-\bar\lambda^{j,\epsilon}_t\wedge\epsilon)\vee\epsilon}\right]dt \\
\forall i\in\{2,\dots, n-1 \}\text{, }  d \bar\lambda^{i,\epsilon}_t =&\ 2\sqrt{\bar\lambda^{i,\epsilon}_t(1-\bar\lambda^{i,\epsilon}_t)}dB^i_t\\
&+\Biggl\{(1-2\bar\lambda^{i,\epsilon}_t)\left[2-0\vee\frac{2}{\sqrt{\epsilon}}\left(\sqrt{1-\bar\lambda^{i,\epsilon}_t}-\frac{\sqrt{\epsilon}}{2}\right)\wedge1 -0\vee\frac{2}{\sqrt{\epsilon}}\left(\sqrt{\bar\lambda^{i,\epsilon}_t}-\frac{\sqrt{\epsilon}}{2}\right)\wedge1 \right]\\
&+\beta\Biggl[p-n+1-(p+q)\bar\lambda^{i,\epsilon}_t+2\frac{(1-\bar\lambda^{i,\epsilon}_t)\wedge\frac{\epsilon}{2}}{\frac{\epsilon}{2}}\left(\frac{\bar\lambda^{i,\epsilon}_t(1-\bar\lambda^{1,\epsilon}_t)}{(\bar\lambda^{i,\epsilon}_t-\bar\lambda^{1,\epsilon}_t)\vee\epsilon}-\frac{\bar\lambda^{i,\epsilon}_t(1-\bar\lambda^{n,\epsilon}_t)}{(\bar\lambda^{n,\epsilon}_t-\bar\lambda^{i,\epsilon}_t)\vee\epsilon}\right)\\
&+2\sum_{2\leq j\leq n-1,j\neq i}\frac{\bar\lambda^{i,\epsilon}_t(1-\bar\lambda^{j,\epsilon}_t)}{\bar\lambda^{i,\epsilon}_t-\bar\lambda^{j,\epsilon}_t}\Biggr]\Biggl\}dt\\
&0\leq \bar\lambda^{1,\epsilon}_t\leq1,\ 0\leq \bar\lambda^{n,\epsilon}_t\leq1 \text{ and } 0\leq \bar\lambda^{2,\epsilon}_t<\dots< \bar\lambda^{n-1,\epsilon}_t\leq1, \text{ a.s.,  } dt-a.e.
\end{aligned}\right.
\tag{$\bar{J} _\epsilon$}
\end{equation}
    These systems are built such that : 
    \begin{eqnarray*}
        \text{(\ref{eq:Jhat})} &\text{  coincides with \eqref{eq:main} on } &\{t, \hat\lambda^{1,\epsilon}_t\geq\frac{\epsilon}{2} \text{ and } \hat\lambda^{n,\epsilon}_t\leq1-\frac{\epsilon}{2}\},\\
        \text{(\ref{eq:Jtilde})} & \text{  coincides with \eqref{eq:main} on } &\{t, \tilde\lambda^{1,\epsilon}_t\leq \epsilon \text{ and } \tilde\lambda^{2,\epsilon}_t-\tilde\lambda^1_t\geq\epsilon \text{ and } \tilde\lambda^{n,\epsilon}_t\leq1-\frac{\epsilon}{2}\},\\
        \text{(\ref{eq:Jcheck})} &\text{  coincides with \eqref{eq:main} on } &\{t, \check\lambda^{1,\epsilon}_t\geq\frac{\epsilon}{2} \text{ and } \check\lambda^{n,\epsilon}_t\geq1-\epsilon\text{ and } \check\lambda^{n,\epsilon}_t-\check\lambda^{n-1,\epsilon}_t\geq\epsilon\},\\
        \text{(\ref{eq:Jbar})} & \text{  coincides with \eqref{eq:main} on } &\{t, \bar\lambda^{1,\epsilon}_t\leq \epsilon \text{ and } \bar\lambda^{2,\epsilon}_t-\bar\lambda^1_t\geq\epsilon  \text{ and }\bar\lambda^{n,\epsilon}_t\geq1-\epsilon\text{ and } \bar\lambda^{n,\epsilon}_t-\bar\lambda^{n-1,\epsilon}_t\geq\epsilon\} .
    \end{eqnarray*}

  Lemmas \ref{lemma:Jhat}, \ref{lemma:Jtilde}, \ref{lemma:Jcheck} and \ref{lemma:Jbar} give  the existence of global pathwise unique strong solutions to (\ref{eq:Jhat}), (\ref{eq:Jtilde}), (\ref{eq:Jcheck}) and (\ref{eq:Jbar}) with any  random initial condition with ordered non-negative coordinates and independent from the driving Brownian motion.

  For $\xi\in[0,1]^n$ deterministic with ordered coordinates, let $\hat\Lambda^{\epsilon,T,\xi}=(\hat{\lambda}_t^{1,\epsilon,T,\xi},\dots,\hat\lambda_t^{1,\epsilon,T,\xi})_t$, $\tilde\Lambda^{\epsilon,T,\xi}=(\tilde{\lambda}_t^{1,\epsilon,T,\xi},\dots,\tilde\lambda_t^{1,\epsilon,T,\xi})_t$, $\check\Lambda^{\epsilon,T,\xi}=(\check{\lambda}_t^{1,\epsilon,T,\xi},\dots,\check\lambda_t^{1,\epsilon,T,\xi})_t$ and $\bar\Lambda^{\epsilon,T,\xi}=(\bar{\lambda}_t^{1,\epsilon,T,\xi},\dots,\bar\lambda_t^{1,\epsilon,T,\xi})_t$ denotes the process solution to respectively (\ref{eq:Jhat}), (\ref{eq:Jtilde}), (\ref{eq:Jcheck}) and (\ref{eq:Jbar}) on $[T,+\infty)$ starting from  $\xi$  at time $T$   and equal to $0$ on $(-\infty,T)$. 

  Let us note $d$ the distance on the set $\Delta=\{\wedge, \sim, \vee,-\}$ defined by 
  \begin{equation*}
      \begin{split}
          d(\sim,\vee) = d(\wedge,-) = 2,\\
          d(\wedge,\sim) = d(\wedge,\vee) = d(-,\sim) = d(-,\vee)=1,
      \end{split}
  \end{equation*}
  which kind be more easily visualised in Figure \ref{figure:d}.

\begin{figure}
  \centering
  \includegraphics[scale=0.5]{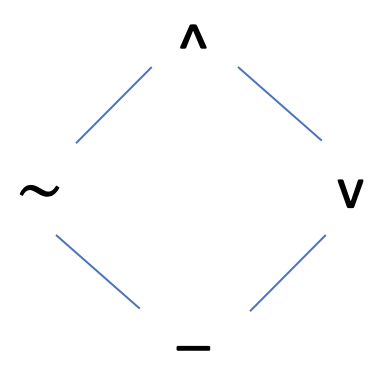}
  \caption{Graphical representation of $(\Delta,d)$.}\label{figure:d}
\end{figure}  

  Let us moreover define the application defined on a subset of $\Delta\times\Delta\times[0,+\infty)\times [0,1]^n$ by

  \begin{equation*}
      \mathsf{T}(symb, symb', t, x)=\left\{\begin{aligned}
          &\inf\left\{t'\geq t : \left(\overset{symb}{\lambda}\right)^{1,\epsilon,t,x}\leq\frac{\epsilon}{2}\right\} \text{ if }(symb,symb')\in\{(\wedge,\sim),(\vee,-)\},\\
          &\inf\left\{t'\geq t : \left(\overset{symb}{\lambda}\right)^{n,\epsilon,t,x}\geq1-\frac{\epsilon}{2}\right\} \text{ if }(symb,symb')\in\{(\wedge,\vee),(\sim,-)\},\\
          &\inf\left\{t'\geq t : \left(\overset{symb}{\lambda}\right)^{n,\epsilon,t,x}\leq1-\epsilon\right\} \text{ if }(symb,symb')\in\{(\vee,\wedge),(-,\sim)\},\\
          &\inf\left\{t'\geq t : \left(\overset{symb}{\lambda}\right)^{1,\epsilon,t,x}\geq\epsilon\right\} \text{ if }(symb,symb')\in\{(\sim,\wedge),(-,\vee)\}.
      \end{aligned}\right.
  \end{equation*}

  We define by induction

  \begin{align*}
      \tau^\epsilon_0 & =  0,\\
      symb^\epsilon_0& = \left\{\begin{aligned}
          \ \wedge & \ \text{ on } \{\lambda^1_0\geq\epsilon \text{ and } \lambda^n_0\leq1-\epsilon\},\\
          \ \sim & \ \text{ on } \{\lambda^1_0<\epsilon \text{ and } \lambda^n_0\leq1-\epsilon\},\\
          \ \vee &\ \text{ on } \{\lambda^1_0\geq\epsilon \text{ and } \lambda^n_0>1-\epsilon\},\\
          \ - &\ \text{ on } \{\lambda^1_0<\epsilon \text{ and } \lambda^n_0>1-\epsilon\},
      \end{aligned}\right.\\
      \tau^\epsilon_1 & =  \bigwedge_{symb\in\Delta,d(symb,symb_0^\epsilon)=1}\mathsf{T}(symb_0^\epsilon,symb,0,\Lambda_0),\\
      \text{ if }\sharp\{symb \in \Delta\backslash\{symb^\epsilon_0\} &\text{ such that } \tau^\epsilon_1  =  \mathsf{T}(symb_0^\epsilon,symb,0,\Lambda_0)\}=1 \text{ then } symb_1^\epsilon \text{ is the only element of that set.}\\ \text{Otherwise, } symb_1^\epsilon &\text{ is the only element of }\Delta \text{ such that } d(symb_0^\epsilon,symb_1^\epsilon)=2,\\ 
      \mathcal{X}^{(1)}_t &= \left(\overset{symb^\epsilon_0}{\Lambda}\right)^{\epsilon,0,\Lambda_0}\mathds{1}_{\left\{0\leq t\leq\tau^\epsilon_1\right\}} \text{ for all }t\in\R,\\
      \tau^\epsilon_2 & =  \bigwedge_{symb\in\Delta,d(symb,symb_1^\epsilon)=1}\mathsf{T}\left(symb_1^\epsilon,symb,\tau^\epsilon_1,\mathcal{X}^{(1)}_{\tau^\epsilon_1}\right),\\
      \mathcal{X}^{(2)}_t &= \left(\overset{symb^\epsilon_1}{\Lambda}\right)^{\epsilon,\tau^\epsilon_1,\mathcal{X}^{(1)}_{\tau^\epsilon_1}}\mathds{1}_{\left\{\tau^\epsilon_1< t\leq\tau^\epsilon_2\right\}} \text{ for all }t\in\R,\\
      &\ \vdots\\
      \tau^\epsilon_{i+1} & =  \bigwedge_{symb\in\Delta,d(symb,symb_i^\epsilon)=1}\mathsf{T}\left(symb_i^\epsilon,symb,\tau^\epsilon_i,\mathcal{X}^{(i)}_{\tau^\epsilon_i}\right),\\
      \text{ if }\sharp\{symb \in \Delta\backslash\{symb^\epsilon_i\} &\text{ such that } \tau^\epsilon_{i+1}  =  \mathsf{T}(symb_i^\epsilon,symb,\tau^\epsilon_i,\mathcal{X}^{(i)}_{\tau^\epsilon_i})\}=1 \text{ then } symb_{i+1}^\epsilon \text{ is the only element of that set.}\\ 
      \text{Otherwise, } symb_{i+1}^\epsilon &\text{ is the only element of }\Delta \text{ such that } d(symb_i^\epsilon,symb_{i+1}^\epsilon)=2,\\ 
      \mathcal{X}^{(i+1)}_t &= \left(\overset{symb^\epsilon_i}{\Lambda}\right)^{\epsilon,\tau^\epsilon_i,\mathcal{X}^{(i)}_{\tau^\epsilon_i}}\mathds{1}_{\left\{\tau^\epsilon_i< t\leq\tau^\epsilon_{i+1}\right\}} \text{ for all }t\in\R,\\
      &\ \vdots
  \end{align*}
  and as for all $i \in\mathbb{N}$, the $\tau_i^\epsilon$ defined before are stopping times for the filtration $(\mathcal{F}_t)_{t\ge0}$, the random vectors $\mathds{1}_{\{\tau^{\epsilon}_i<+\infty\}}\mathcal{X}^{(i)}_{\tau_i^\epsilon}$ are $\mathcal{F}_{\tau_i^\epsilon}$-measurable, the construction makes sense.

  We finally define for all $\epsilon>0$ and $t\geq0$   :
  \begin{equation*}
    \mathcal{Z}^\epsilon_t = \sum_{i=1}^{+\infty}\mathcal{X}^{(i)}_{t}.
  \end{equation*}

  Let us verify that the process $(\mathcal{Z}^\epsilon_t)_t$ is defined globally. To do so, we show that there is no accumulation of the stopping times $\tau^\epsilon_i$ by first separating them between the ones copping with the singularity at $0$ and the ones copping with the singularity at $1$. 

  On $symb^\epsilon_0\in\{\vee,\wedge\}$, let us define $(\sigma^{0,\epsilon}_i)_{i\geq0}$  as  the subsequence of $(\tau^\epsilon_i)_{i\geq0}$ corresponding to the stopping times involving the position of the smallest particle of the considered system regarding $\epsilon$ and $\frac{\epsilon}{2}$ and the stopping times  that can be attributed both to a change of position of the smallest particle and of the biggest particle. 
  This means that on $symb^\epsilon_0\in\{\vee,\wedge\}$, we can write $(\sigma^{0,\epsilon}_i)_{i\geq0}=(\tau^{\epsilon}_{\varphi(i)})_{i\geq0}$ where $\varphi$ is the only bijective increasing map from $ \N $ to $\{j\in\N^*, (symb^\epsilon_{j-1}, symb^\epsilon_j)\in\{(\wedge, \sim),(\vee,-), (\sim,\wedge), (-,\vee),(\wedge,-),(-,\wedge), (\sim,\vee),(\vee,\sim)\}$, where $\N=\{0,1,2,\dots\}$ and $\N^*=\N\backslash\{0\}$.

  On $symb^\epsilon_0\in\{\sim,-\}$, let us define $\sigma^{0,\epsilon}_0=0$ and $(\sigma^{0,\epsilon}_i)_{i\geq1}$  as  the subsequence of $(\tau^\epsilon_i)_{i\geq0}$ corresponding to the stopping times involving the position of the smallest particle of the considered system regarding $\epsilon$ and $\frac{\epsilon}{2}$ and the stopping times  that be attributed both to a change of position of the smallest particle and of the biggest particle. 
  This means that on $symb^\epsilon_0\in\{\sim,-\}$, we can write $(\sigma^{0,\epsilon}_i)_{i\geq1}=(\tau^{\epsilon}_{\varphi(i)})_{i\geq0}$ where $\varphi$ is the only bijective increasing map from $ \N^* $ to $\{j\in\N^*, (symb^\epsilon_{j-1}, symb^\epsilon_j)\in\{(\wedge, \sim),(\vee,-), (\sim,\wedge), (-,\vee),(\wedge,-),(-,\wedge), (\sim,\vee),(\vee,\sim)\}$.

   On $symb^\epsilon_0\in\{\sim,\wedge\}$, let us define the sequence $(\sigma^{1,\epsilon}_i)_{i\geq0}$ as the subsequence of $(\tau^\epsilon_i)_{i\geq0}$ corresponding to the stopping times involving the position of the biggest particle of the considered system regarding $1-\epsilon$ and $1-\frac{\epsilon}{2}$.
   This means that on $symb^\epsilon_0\in\{\sim,\wedge\}$, we can write $(\sigma^{1,\epsilon}_i)_{i\geq0}=(\tau^{\epsilon}_{\varphi(i)})_{i\geq0}$ where $\varphi$ is the only bijective increasing map from $ \N $ to $\{j\in\N^*, (symb^\epsilon_{j-1}, symb^\epsilon_j)\in\{(\wedge, \vee),(\sim,-), (\vee,\wedge), (-,\sim)\}$.
   
   On $symb^\epsilon_0\in\{\vee,-\}$, let us define the sequence $(\sigma^{1,\epsilon}_i)_{i\geq0}$ as $\sigma^{1,\epsilon}_0=0$ and  $(\sigma^{1,\epsilon}_i)_{i\geq1}$ as the subsequence of $(\tau^\epsilon_i)_{i\geq0}$ corresponding to the stopping times involving the position of the biggest particle of the considered system regarding $1-\epsilon$ and $1-\frac{\epsilon}{2}$. 
   This means that on $symb^\epsilon_0\in\{\vee,-\}$, we can write $(\sigma^{1,\epsilon}_i)_{i\geq1}=(\tau^{\epsilon}_{\varphi(i)})_{i\geq0}$ where $\varphi$ is the only bijective increasing map from $ \N^* $ to $\{j\in\N^*, (symb^\epsilon_{j-1}, symb^\epsilon_j)\in\{(\wedge, \vee),(\sim,-), (\vee,\wedge), (-,\sim)\}$.

  We have for $i\ge1$ :

  \begin{itemize}
      \item for all $t\in(\sigma^{0,\epsilon}_{2i+1},\sigma^{0,\epsilon}_{2i+2}]$,
      \begin{eqnarray*}
        d\mathcal{Z}^{\epsilon,1}_t &\leq&  2\sqrt{\mathcal{Z}^{\epsilon,1}_t(1-\mathcal{Z}^{\epsilon,1}_t)}dB^1_t+\beta\left[p-n+1-(p+q)\mathcal{Z}^{\epsilon,1}_t\right]dt, 
      \end{eqnarray*}
      and thus, according to the pathwise comparison theorem of Ikeda and Watanabe (that we recall in Theorem \ref{Ikeda} below)
      \[\mathcal{Z}^{\epsilon,1}_t\leq r_{t-\sigma^{0,\epsilon}_{2i+1}}^{0,i},\]
      where for all $t\geq 0$
      \begin{equation*}
          r_t^{0,i} = \frac{\epsilon}{2}+2\int_{0}^{t}\sqrt{r_s^{0,i}(1-r_s^{0,i})}dB^1_{s+\sigma^{0,\epsilon}_{2i+1}}+\beta (p-n+1) t -\beta(p+q)\int_{0}^{t}r_s^{0,i}ds,
      \end{equation*}
      which is a real Jacobi  process. The times $\sigma^{0,\epsilon}_{2i+2}-\sigma^{0,\epsilon}_{2i+1}$ for all $i\ge1$  are larger than the time interval for the  Jacobi processes  $r^{0,i}$  to go from $\frac{\epsilon}{2}$ to $\epsilon$. Moreover, the times for the $r^{0,i}$ processes to go from $\frac{\epsilon}{2}$ to $\epsilon$ are iid positive random variables. Consequently, there is no accumulation of the stopping times  $\sigma^{0,\epsilon}_{i}$ which go to infinity as $i\rightarrow\infty$.  
      \item for all $t\in(\sigma^{1,\epsilon}_{2i+1},\sigma^{1,\epsilon}_{2i+2}]$,
      \begin{eqnarray*}
        d\mathcal{Z}^{\epsilon,n}_t &\geq&  2\sqrt{\mathcal{Z}^{\epsilon,n}_t(1-\mathcal{Z}^{\epsilon,n}_t)}dB^n_t+\beta\left[p-n+1-(p+q)\mathcal{Z}^{\epsilon,n}_t\right]dt, 
      \end{eqnarray*}
      and thus
      \[\mathcal{Z}^{\epsilon,n}_t\geq r_{t-\sigma^{1,\epsilon}_{2i+1}}^{1,i},\]
      where for all $t\geq 0$
      \begin{equation*}
          r_t^{1,i} = 1-\frac{\epsilon}{2}+2\int_{0}^{t}\sqrt{r_s^{1,i}(1-r_s^{1,i})}dB^n_{s+\sigma^{1,\epsilon}_{2i+1}}+\beta (p-n+1) t -\beta(p+q)\int_{0}^{t}r_s^{1,i}ds,
      \end{equation*}
      which is a real Jacobi  process. The times $\sigma^{1,\epsilon}_{2i+2}-\sigma^{1,\epsilon}_{2i+1}$ for all $i\ge1$  are larger than the time interval for the  Jacobi processes  $r^{1,i}$  to go from $1-\frac{\epsilon}{2}$ to $1-\epsilon$. Moreover, the times for the $r^{1,i}$ processes to go from $1-\frac{\epsilon}{2}$ to $1-\epsilon$ are iid positive random variables. Consequently, there is no accumulation of the stopping times  $\sigma^{1,\epsilon}_{i}$ which go to infinity as $i\rightarrow\infty$.  
  \end{itemize}
  The stochastic process $\mathcal{Z}^\epsilon$ is thus defined globally.

  We recall that the SDEs (\ref{eq:Jhat}), (\ref{eq:Jtilde}), (\ref{eq:Jcheck}), (\ref{eq:Jbar}) respectively coincide with the system of SDEs \eqref{eq:main} on  $\{t, \hat\lambda^{1,\epsilon}_t\geq\frac{\epsilon}{2} \text{ and } \hat\lambda^{n,\epsilon}_t\leq1-\frac{\epsilon}{2}\}$, $\{t, \tilde\lambda^{1,\epsilon}_t\leq \epsilon \text{ and } \tilde\lambda^{2,\epsilon}_t-\tilde\lambda^1_t\geq\epsilon \text{ and } \tilde\lambda^{n,\epsilon}_t\leq1-\frac{\epsilon}{2}\}$, $\{t, \check\lambda^{1,\epsilon}_t\geq\frac{\epsilon}{2} \text{ and } \check\lambda^{n,\epsilon}_t\geq1-\epsilon\text{ and } \check\lambda^{n,\epsilon}_t-\check\lambda^{n-1,\epsilon}_t\geq\epsilon\}$ and $\{t, \bar\lambda^{1,\epsilon}_t\leq \epsilon \text{ and } \bar\lambda^{2,\epsilon}_t-\bar\lambda^1_t\geq\epsilon  \text{ and }\bar\lambda^{n,\epsilon}_t\geq1-\epsilon\text{ and } \bar\lambda^{n,\epsilon}_t-\bar\lambda^{n-1,\epsilon}_t\geq\epsilon\}$. On the other hand, $(\mathcal{Z}^\epsilon_t)_t$ evolves according to those four SDEs, and verify $\mathcal{Z}^{\epsilon,1}_t\geq\frac{\epsilon}{2}$ on $[\tau^\epsilon_i,\tau^\epsilon_{i+1}]$ for $i$ such that $symb^\epsilon_i\in\{\wedge,\vee\}$ and $\mathcal{Z}^{\epsilon,n}_t\leq1-\frac{\epsilon}{2}$ on $[\tau^\epsilon_i,\tau^\epsilon_{i+1}]$ for $i$ such that $symb^\epsilon_i\in\{\wedge,\sim\}$. By induction on $i$, $(\mathcal{Z}^\epsilon_t)_t$ is a solution to \eqref{eq:main} until

  \begin{align*}
    \inf&\left\{t\in\bigcup_{i\in\mathbb{N}, symb^\epsilon_i\in\{\sim,-\}}[\tau^\epsilon_{i},\tau^\epsilon_{i+1}], \mathcal{Z}^{2,\epsilon}_t-\mathcal{Z}^{1,\epsilon}_t\leq\epsilon\right\}\wedge\inf\left\{t\in\bigcup_{i\in\mathbb{N}, symb^\epsilon_i\in\{\vee,-\}}[\tau^\epsilon_{i},\tau^\epsilon_{i+1}], \mathcal{Z}^{n,\epsilon}_t-\mathcal{Z}^{n-1,\epsilon}_t\leq\epsilon\right\}\\
    &\geq\inf\left\{t\geq0 :  \left(\mathcal{Z}^{1,\epsilon}_t\leq\epsilon \text{ and }\mathcal{Z}^{2,\epsilon}_t-\mathcal{Z}^{1,\epsilon}_t\leq\epsilon\right) \text{ or } \left(\mathcal{Z}^{n,\epsilon}_t\geq1-\epsilon \text{ and }\mathcal{Z}^{n,\epsilon}_t-\mathcal{Z}^{n-1,\epsilon}_t\leq\epsilon\right)\right\} =:\zeta_\epsilon.
\end{align*}
From Lemmas \ref{lemma:Jhat}, \ref{lemma:Jtilde}, \ref{lemma:Jcheck} and \ref{lemma:Jbar} we have : 
\begin{eqnarray}
    \label{eq:multiple-collision-hat}
    &&\mathbb{P}\{\exists i\in\{j\in\N, symb^\epsilon_j=\wedge\},\exists t\in(\tau_{i}^\epsilon, \tau_{i+1}^\epsilon]:\mathcal{Z}^{k,\epsilon}_t=\mathcal{Z}^{k+1,\epsilon}_t \text{ and } \mathcal{Z}^{l,\epsilon}_t=\mathcal{Z}^{l+1,\epsilon}_t \text{ for some } 0\leq k<l\leq n\}=0\quad\\
    \label{eq:multiple-collision-tilde}
    &&\mathbb{P}\{\exists i\in\{j\in\N, symb^\epsilon_j=\ \sim\},\exists t\in(\tau_{i}^\epsilon, \tau_{i+1}^\epsilon]:\mathcal{Z}^{k,\epsilon}_t=\mathcal{Z}^{k+1,\epsilon}_t \text{ and } \mathcal{Z}^{l,\epsilon}_t=\mathcal{Z}^{l+1,\epsilon}_t \text{ for some } 2\leq k<l\leq n\}=0\\
    \label{eq:multiple-collision-check}
    &&\mathbb{P}\{\exists i\in\{j\in\N, symb^\epsilon_j=\vee\},\exists t\in(\tau_{i}^\epsilon, \tau_{i+1}^\epsilon]:\mathcal{Z}^{k,\epsilon}_t=\mathcal{Z}^{k+1,\epsilon}_t \text{ and } \mathcal{Z}^{l,\epsilon}_t=\mathcal{Z}^{l+1,\epsilon}_t \text{ for some } 0\leq k<l\leq n-2\}=0\\
    \label{eq:multiple-collision-bar}
    &&\mathbb{P}\{\exists i\in\{j\in\N, symb^\epsilon_j=-\},\exists t\in(\tau_{i}^\epsilon, \tau_{i+1}^\epsilon]:\mathcal{Z}^{k,\epsilon}_t=\mathcal{Z}^{k+1,\epsilon}_t \text{ and } \mathcal{Z}^{l,\epsilon}_t=\mathcal{Z}^{l+1,\epsilon}_t \text{ for some } 2\leq k<l\leq n-2\}=0
\end{eqnarray}
where by convention $\mathcal{Z}^{0,\epsilon}\equiv0$ and $\mathcal{Z}^{n+1,\epsilon}\equiv1$.

On the time intervals $[\tau_{i}^\epsilon\wedge\zeta_\epsilon, \tau_{i+1}^\epsilon\wedge\zeta_\epsilon]$ for $i\in\{j\in\N, symb^\epsilon_j\in\{\sim,-\}\}$ , we have $ \mathcal{Z}^{2,\epsilon}_t-\mathcal{Z}^{1,\epsilon}_t\geq\epsilon$ and on the time intervals $[\tau_{i}^\epsilon\wedge\zeta_\epsilon, \tau_{i+1}^\epsilon\wedge\zeta_\epsilon]$ for $i\in\{j\in\N, symb^\epsilon_j\in\{\vee,-\}\}$ , we have $ \mathcal{Z}^{n,\epsilon}_t-\mathcal{Z}^{n-1,\epsilon}_t\geq\epsilon$. This together with (\ref{eq:multiple-collision-hat},\ref{eq:multiple-collision-tilde},\ref{eq:multiple-collision-check},\ref{eq:multiple-collision-bar}) allows to conclude :
\begin{equation}
\label{eq:multiple-collision-Z}
    \mathbb{P}\Big\{\exists t\in (0,\zeta_\epsilon]: \mathcal{Z}^{i,\epsilon}_t = \mathcal{Z}^{i+1,\epsilon}_t \text{ and } \mathcal{Z}^{j,\epsilon}_t = \mathcal{Z}^{j+1,\epsilon}_t \text{ for some }1\leq i<j\leq n-1\Big\}=0.
\end{equation}

As the solutions to equation \eqref{eq:main} are pathwise unique (see Lemma \ref{lemma:pathwise_uniqueness}), for $n\in\mathbb{N}^*$, the processes $\mathcal{Z}^{\frac{1}{n}}$ and $\mathcal{Z}^{\frac{1}{n+1}}$ coincide on $\left[0,\zeta_{\frac{1}{n}}\wedge\zeta_{\frac{1}{n+1}}\right]$. Thus, $\zeta_{\frac{1}{n}}\wedge\zeta_{\frac{1}{n+1}}=\zeta_{\frac{1}{n}}$ and the sequence $\left(\zeta_{\frac{1}{n}}\right)_{n\in\mathbb{N}^*}$  is non-decreasing. Moreover, for all $n\in\mathbb{N}^*$, $\mathcal{Z}^{\frac{1}{n}}$  verifies (\ref{eq:multiple-collision-Z}). Consequently,  we can define for all $t\in[0,\underset{\epsilon\rightarrow0}{\lim}\zeta_\epsilon)$
\begin{equation}
\label{solution}
    \Lambda_t =\mathcal{Z}^1_{t}\mathds{1}_{\left\{0\leq t\leq\zeta_{1}\right\}}+ \sum_{n\geq1}\mathcal{Z}^{\frac{1}{n+1}}_t\mathds{1}_{\left\{\zeta_{\frac{1}{n}}<t\leq\zeta_{\frac{1}{n+1}}\right\}},
\end{equation}
which is a solution to the system of SDEs \eqref{eq:main} on $[0,\underset{\epsilon\rightarrow0}{\lim}\zeta_\epsilon)$ verifying $(iii)$ of Theorem \ref{thm:existence}.

Finally,
as the solutions to \eqref{eq:main} are pathwise unique   (Lemma \ref{lemma:pathwise_uniqueness}), we can apply the Yamada-Watanabe theorem (see for instance \cite[Theorem 1.7 p.368]{MR1725357}) to deduce the existence of strong solutions to the equation.

\paragraph{}
  
   Since on $\left\{\zeta_\epsilon<+\infty\right\}$ we have $\lambda^1_{\zeta_\epsilon}+\lambda^2_{\zeta_\epsilon}= 2\lambda^1_{\zeta_\epsilon}+\lambda^2_{\zeta_\epsilon}-\lambda^1_{\zeta_\epsilon}\leq3\epsilon$, or $\left(1-\lambda^n_{\zeta_\epsilon}\right)+\left(1-\lambda^{n-1}_{\zeta_\epsilon}\right)= 2\left(1-\lambda^n_{\zeta_\epsilon}\right)+\left(1-\lambda^{n-1}_{\zeta_\epsilon}\right)-\left(1-\lambda^n_{\zeta_\epsilon}\right)\leq3\epsilon$, we have on $\left\{\underset{\epsilon\rightarrow0}{\lim}\zeta_\epsilon<+\infty\right\}$ : $\underset{t\in[0,\underset{\epsilon\rightarrow0}{\lim}\zeta_\epsilon)}{\inf}\lambda^1_t+\lambda^2_t=0$ or $\underset{t\in[0,\underset{\epsilon\rightarrow0}{\lim}\zeta_\epsilon)}{\sup}\lambda^n_t+\lambda^{n-1}_t=2$.
   On $\left\{\underset{\epsilon\rightarrow0}{\lim}\zeta_\epsilon<+\infty\right\}\cap\left\{\underset{t\in[0,\underset{\epsilon\rightarrow0}{\lim}\zeta_\epsilon)}{\inf}\lambda^1_t+\lambda^2_t=0\right\}$ we use Proposition \ref{prop:multiplecollision} $(i)$  with $k=2$ to conclude that 
   $$\underset{\epsilon\rightarrow0}{\lim}\inf\left\{t\geq0 :  \mathcal{Z}^{1,\epsilon}_t\leq\epsilon \text{ and }\mathcal{Z}^{2,\epsilon}_t-\mathcal{Z}^{1,\epsilon}_t\leq\epsilon\right\} =+\infty \text{ when } p-n+1\geq\frac{1}{\beta}-1.$$ 

   Likewise, on $\left\{\underset{\epsilon\rightarrow0}{\lim}\zeta_\epsilon<+\infty\right\}\cap\left\{\underset{t\in[0,\underset{\epsilon\rightarrow0}{\lim}\zeta_\epsilon)}{\sup}\lambda^n_t+\lambda^{n-1}_t=2\right\}$
   we  use Proposition \ref{prop:multiplecollision} $(ii)$ with $k=2$ to conclude that 
   $$\underset{\epsilon\rightarrow0}{\lim}\inf\left\{t\geq0 :  \mathcal{Z}^{n,\epsilon}_t\geq1-\epsilon \text{ and }\mathcal{Z}^{n,\epsilon}_t-\mathcal{Z}^{n-1,\epsilon}_t\leq\epsilon\right\} =+\infty \text{ when } q-n+1\geq\frac{1}{\beta}-1.$$ Consequently, when $p\wedge q -n+1\geq \frac{1}{\beta}-1$,
   \[\underset{\epsilon\rightarrow0}{\lim}\zeta_\epsilon =+\infty,\]

   which is $(i)$ from Theorem \ref{thm:existence}. 
\end{proof}
\begin{lemma}
    \label{lemma:Jhat}
    Let us assume $p\geq q$ and $q-n+1>0$.
    The system of  SDEs
    \begin{equation}
        \label{eq:Jhat}
        \begin{split}
            d \hat\lambda^{i,\epsilon}_t = 2\sqrt{\hat\lambda^{i,\epsilon}_t(1-\hat\lambda^{i,\epsilon}_t)}dB^i_t &+\Biggl\{(1-2\hat\lambda^{i,\epsilon}_t)\left[2-0\vee\frac{2\sqrt{2}}{\sqrt{\epsilon}}\left(\sqrt{1-\hat\lambda^{i,\epsilon}_t}-\frac{\sqrt{\epsilon}}{2\sqrt{2}}\right)\wedge1 -0\vee\frac{2\sqrt{2}}{\sqrt{\epsilon}}\left(\sqrt{\hat\lambda^{i,\epsilon}_t}-\frac{\sqrt{\epsilon}}{2\sqrt{2}}\right)\wedge1 \right]\\
            &+\beta\left[p-(p+q)\hat\lambda^{i,\epsilon}_t+\sum_{j\neq i}\frac{\hat\lambda^{i,\epsilon}_t(1-\hat\lambda^{j,\epsilon}_t)+\hat\lambda^{j,\epsilon}_t(1-\hat\lambda^{i,\epsilon}_t)}{\hat\lambda^{i,\epsilon}_t-\hat\lambda^{j,\epsilon}_t}\right]\Biggr\}dt \\
            &0\leq \hat\lambda^{1,\epsilon}_t<\dots< \hat\lambda^{n,\epsilon}_t\leq1, \text{ a.s.,  } dt-a.e.
        \end{split}\tag{$\hat{J} _\epsilon$}
        \end{equation}
    
      has a global pathwise unique strong solution $(\hat{\lambda}^{1,\epsilon}_t,\dots,\hat{\lambda}^{1,\epsilon}_t)_{t}$ starting from any random  initial condition $\Lambda_0=(\lambda^1_0,\dots,\lambda^n_0)$ independent from $\mathbf{B}$ such that $0\leq \lambda^1_0\leq\dots\leq\lambda^n_0\leq1$ a.s.
      
      Moreover,
      \begin{equation}
      \label{eq:Jhatnonmultiplecollision}
          \mathbb{P}\{\exists t>0:\hat\lambda^{i,\epsilon}_t=\hat\lambda^{i+1,\epsilon}_t \text{ and } \hat\lambda^{j,\epsilon}_t=\hat\lambda^{j+1,\epsilon}_t \text{ for some } 0\leq i<j\leq n\}=0,
      \end{equation}
      where by convention $\hat\lambda^{0,\epsilon}\equiv0$ and $\hat\lambda^{n+1,\epsilon}\equiv1$.
\end{lemma}

    \begin{proof}
    Let us consider  $\mathcal{F}_t=\sigma\left(\left(\arcsin\left(\sqrt{\lambda^1_0}\right),\dots,\arcsin\left(\sqrt{\lambda^n_0}\right)\right), \left(\mathbf{B}_s \right)_{s \le t} \right)$    and the system of SDEs defined by 

    \begin{equation}
    \label{eq:JhatGrad}
        \begin{split}
            d\hat\phi^{\epsilon,i}_t=&dB^i_t+\Biggl\{\frac{\beta(p-q)}{2}\mathrm{cot}\hat\phi^{\epsilon,i}_t\\
            &+\left(\beta(q-n+1)+\left[1-0\vee\frac{2\sqrt{2}}{\sqrt{\epsilon}}\left(\cos\hat\phi^{\epsilon,i}_t-\frac{\sqrt{\epsilon}}{2\sqrt{2}}\right)\wedge1 -0\vee\frac{2\sqrt{2}}{\sqrt{\epsilon}}\left(\sin\hat\phi^{\epsilon,i}_t-\frac{\sqrt{\epsilon}}{2\sqrt{2}}\right)\wedge1 \right]\right)\mathrm{cot}(2\hat\phi^{\epsilon,i}_t)\\
            &+\frac{\beta}{2}\sum_{j\ne i}\left[\mathrm{cot}(\hat\phi^{\epsilon,i}_t+\hat\phi^{\epsilon,j}_t)+\mathrm{cot}(\hat\phi^{\epsilon,i}_t-\hat\phi^{\epsilon,j}_t)\right]\Biggr\}dt\\
            &0\leq \hat\phi^{\epsilon,1}_t<\dots< \hat\phi^{\epsilon,n}_t\le\frac{\pi}{2}, \text{  a.s. } dt-\text{a.e.,}
        \end{split}   \tag{$\hat{J}_\epsilon^\phi$}
    \end{equation}
      with  random initial condition $\left(\arcsin\left(\sqrt{\lambda^1_0}\right),\dots,\arcsin\left(\sqrt{\lambda^n_0}\right)\right)$ such that $0\leq \arcsin\left(\sqrt{\lambda^1_0}\right)\leq\dots\leq \arcsin\left(\sqrt{\lambda^n_0}\right)\leq\frac{\pi}{2}$.

    We are going to apply  Cepa's multivoque equations theory (\cite{MR1459451}) to conclude that there exists a unique strong solution to (\ref{eq:JhatGrad}). 
    
    To do so, we define
    \begin{eqnarray}
    D&=&\{0<\phi^1<\phi^2<\dots<\phi^n<\frac{\pi}{2}\}\nonumber\\
    \mathcal{V}  &:&  (\phi^1,\dots,\phi^n)\in \mathbb{R}^n \rightarrow 
  \left\{
  \begin{array}{c }\displaystyle
    -\sum_{i=1}^n\Bigg\{\frac{\beta(p-q)}{2}\ln|\sin\phi^i|+\beta\frac{q-n+1}{2}\ln|\sin (2\phi^i)|\\\displaystyle
    \qquad\qquad +\frac{\beta}{4}\sum_{j\neq i}\left(\ln|\sin(\phi^i+\phi^j)|+\ln|\sin(\phi^i-\phi^j)|\right)\Bigg\}\text{ if } x\in D \\
      +\infty \text{ if } x\notin D 
  \end{array}
  \right.\nonumber\\
  g &:&  (\phi^1,\dots,\phi^n)\in \mathbb{R}^n \rightarrow 
      \Biggl(\left[1-0\vee\frac{2\sqrt{2}}{\sqrt{\epsilon}}\left(\cos\phi^1-\frac{\sqrt{\epsilon}}{2\sqrt{2}}\right)\wedge1 -0\vee\frac{2\sqrt{2}}{\sqrt{\epsilon}}\left(\sin\phi^1-\frac{\sqrt{\epsilon}}{2\sqrt{2}}\right)\wedge1 \right]\mathrm{cot}(2\phi^1),\nonumber\\
      &&\qquad\qquad\qquad\qquad\dots,\left[1-0\vee\frac{2\sqrt{2}}{\sqrt{\epsilon}}\left(\cos\phi^n-\frac{\sqrt{\epsilon}}{2\sqrt{2}}\right)\wedge1 -0\vee\frac{2\sqrt{2}}{\sqrt{\epsilon}}\left(\sin\phi^n-\frac{\sqrt{\epsilon}}{2\sqrt{2}}\right)\wedge1 \right]\mathrm{cot}(2\phi^n)\Biggr), \nonumber
    \end{eqnarray}
    to rewrite the system of SDEs on $D$ with $\hat\Phi=(\hat\phi^{1,\epsilon}_t,\dots,\hat\phi^{n,\epsilon}_t)_t$  the following way
    
    \begin{equation}
   d\hat\Phi_t  =  d\mathbf{B}_t+g(\hat\Phi_t)dt-\nabla\mathcal{V}(\hat\Phi_t)dt.\tag{$\hat{J}_\epsilon^\phi$}
    \end{equation}
    Since $g$ is globally Lipschitz and $\mathcal{V}$ is convex, Cépa's multivoque equations theory shows the existence and uniqueness of a strong solution to equation

    \begin{align}
     d\hat{\varphi}_t & =  d\mathbf{B}_t+g(\hat{\varphi}_t)dt-\nabla\mathcal{V}(\hat{\varphi}_t)dt-\nu(\hat{\varphi}_t)dL_t \text{ for all  }t\geq0\label{eq:hatJgradCepa}\tag{$\hat{\mathcal{J}}_\epsilon^\phi$}\\
     & \forall t\geq0,\hat{\varphi}_t\in \Bar D \text{ a.s.}\nonumber\\
     &\hat{\varphi}_0=\left(\arcsin\left(\sqrt{\lambda^1_0}\right),\dots,\arcsin\left(\sqrt{\lambda^n_0}\right)\right),\nonumber
    \end{align}
    where $\hat{\varphi}$  is a continuous adapted to $(\mathcal{F}_t)_{t\geq0}$ process,   $L$ is a continuous non-decreasing adapted to $(\mathcal{F}_t)_{t\geq0}$ process with $L_0=0$  verifying
    \begin{equation}
        L_t= \int_0^t\mathds{1}_{\{\hat{\varphi}_s\in\partial D\}}dL_s,\nonumber
    \end{equation}
    and
     $\nu(x)\in\pi(x)$ ($\pi(x)$ is the set of unitary outward normals to $\partial D$ at $x\in\partial D$). The solution to equation (\ref{eq:hatJgradCepa}) follows the conditions : for all $t>0$ 
    \begin{eqnarray}
    \mathbb{E}\left[\int_0^t\mathds{1}_{\{\hat{\varphi}_s\in\partial D\}}ds\right] & = & 0,\nonumber\\
    \mathbb{E}\left[\int_0^t|\nabla\mathcal{V}(\hat{\varphi}_s)|ds\right] & < & \infty.\nonumber
    \end{eqnarray}

    We apply \cite[Theorem 2.2]{MR1875671} which is an application of Cépa's multivoque equations theory to this kind of SDE and the remark following \cite[Theorem 3.1]{MR1459451} to deduce that (\ref{eq:hatJgradCepa})  has a unique strong solution. To prove that the boundary process $L$ is equal to zero, we just follow the steps of the proof of \cite[Lemma 2.2]{Demni2009BetaJP}, itself coming from the proof of \cite[Lemma 1]{MR2566989} which is an adaptation of the proof \cite[Theorem 2.2, Step 1]{MR1875671} and \cite[Lemma 3.8]{MR1875671}.

    Then, setting with $\hat\lambda^{i,\epsilon}=\sin^2(\varphi^{i,\epsilon})$ for all $i\in\{1,\dots,n\}$ we obtain a global solution to (\ref{eq:Jhat}). Following the  approach used in the proof of Lemma \ref{lemma:Jcheck} below to demonstrate the pathwise uniqueness of a slightly more complicated system \eqref{eq:Jtilde}, we obtain that the solutions to  (\ref{eq:Jhat}) are pathwise unique. The Yamada-Watanabe Theorem (see for instance \cite[Theorem 1.7 p.368]{MR1725357}) allows to conclude that (\ref{eq:Jhat}) has a pathwise unique global strong solution.
    
    Let us now prove $(\ref{eq:Jhatnonmultiplecollision})$.
    
    Let us consider the system of SDEs defined by (\ref{eq:JhatGrad})  with initial condition $0\leq \arcsin\left(\sqrt{\lambda^1_0}\right)\leq\dots\leq \arcsin\left(\sqrt{\lambda^n_0}\right)\leq\frac{\pi}{2}$.

      Let us define for all $\epsilon>0$ and for  $t\geq0$ :   $\Theta(t)=(\theta_1(t),\dots,\theta_n(t))$ with
    \begin{eqnarray}
      \forall i\in\{1,\dots, n \}\text{, } \theta_i(t) & = &\left[1-0\vee\frac{2\sqrt{2}}{\sqrt{\epsilon}}\left(\cos\hat\phi^{\epsilon,i}_t-\frac{\sqrt{\epsilon}}{2\sqrt{2}}\right)\wedge1 -0\vee\frac{2\sqrt{2}}{\sqrt{\epsilon}}\left(\sin\hat\phi^{\epsilon,i}_t-\frac{\sqrt{\epsilon}}{2\sqrt{2}}\right)\wedge1 \right]\mathrm{cot}(2\hat\phi^{\epsilon,i}_t),\nonumber
    \end{eqnarray}
    and for all $t\geq0$
    \begin{equation}
        Z(t)=\exp\left\{-\int_0^{t}\Theta(u)\cdot d{\textbf{B}}_u-\frac{1}{2}\int_0^{t}||\Theta(u)||^2du\right\}. \nonumber
    \end{equation}

    We have for all $i\in\{1,\dots,n\}$,
    \begin{eqnarray}
      \theta_i^2(t) & = &\left(\left[1-0\vee\frac{2\sqrt{2}}{\sqrt{\epsilon}}\left(\cos\hat\phi^{\epsilon,i}_t-\frac{\sqrt{\epsilon}}{2\sqrt{2}}\right)\wedge1 -0\vee\frac{2\sqrt{2}}{\sqrt{\epsilon}}\left(\sin\hat\phi^{\epsilon,i}_t-\frac{\sqrt{\epsilon}}{2\sqrt{2}}\right)\wedge1 \right]\mathrm{cot}(2\hat\phi^{\epsilon,i}_t)\right)^2\nonumber\\
      & \leq & \left(\frac{1-\frac{\epsilon}{4}}{2\sqrt{\frac{\epsilon}{8}\left(1-\frac{\epsilon}{8}\right)}}\right)^2,\nonumber
    \end{eqnarray}   
    where we used the fact that the support of the map $$x\in[0,1]\mapsto\left[1-0\vee\frac{2\sqrt{2}}{\sqrt{\epsilon}}\left(\sqrt{1-x}-\frac{\sqrt{\epsilon}}{2\sqrt{2}}\right)\wedge1 -0\vee\frac{2\sqrt{2}}{\sqrt{\epsilon}}\left(\sqrt{x}-\frac{\sqrt{\epsilon}}{2\sqrt{2}}\right)\wedge1 \right]\in[0,1],$$ is $\left[\frac{\epsilon}{8},1-\frac{\epsilon}{8}\right]$ and that the map $\phi\mapsto \cot(2\phi)=\frac{1-\sin^2(\phi)}{2\sqrt{\sin^2(\phi)(1-\sin^2(\phi))}}$ is monotonous on the interval $\left(0,\frac{\pi}{2}\right)$.

    We thus have
    \begin{eqnarray}
      \mathbb{E}\left[\exp\left\{\frac{1}{2}\int_0^{t}||\Theta(u)||^2du\right\}\right]& <&\infty \text{ for all }t\geq0.\nonumber
    \end{eqnarray} 
    Then, according to Novikov's criterion (see for instance \cite[Proposition 5.12 p.198]{MR1121940}), $Z$ is a $\mathbb{P}$-martingale, and $\mathbb{E}[Z(t)]=1$. Consequently, recalling that   $\mathcal{F}_t=\sigma\left(\left(\arcsin\left(\sqrt{\lambda^1_0}\right),\dots, \arcsin\left(\sqrt{\lambda^n_0}\right)\right), \left(\mathbf{B}_s \right)_{s \le t} \right)$  and defining $\mathbb{Q}$ such   that
    $$\frac{d\mathbb{Q}}{d\mathbb{P}}_{|\mathcal{F}_t}=Z(t),$$ and  for all  $i\in\{1,\dots,n\}$, 
    \begin{eqnarray}
        \hat{ B}^{i}_t  &=& B^i_t+\int_0^{t}\theta_i(s)ds \nonumber \\
        & = &  B^i_t+\int_0^{t}\left[1-0\vee\frac{2\sqrt{2}}{\sqrt{\epsilon}}\left(\cos\hat\phi^{\epsilon,i}_s-\frac{\sqrt{\epsilon}}{2\sqrt{2}}\right)\wedge1 -0\vee\frac{2\sqrt{2}}{\sqrt{\epsilon}}\left(\sin\hat\phi^{\epsilon,i}_s-\frac{\sqrt{\epsilon}}{2\sqrt{2}}\right)\wedge1 \right]\mathrm{cot}(2\hat\phi^{\epsilon,i}_s)ds, \text{ }t\geq 0,\nonumber
    \end{eqnarray}
    ${\hat{\textbf{B}}}=({\hat{ B}^{1}}_t,\dots,{\hat{ B}^{n}}_t)_t$ is a $\mathbb{Q}$- Brownian motion
    according to the Girsanov theorem (see for instance \cite[Proposition 5.4 p.194]{MR1121940}).
    
    Thus,  (\ref{eq:JhatGrad}) can be rewritten in terms of $\hat{\mathbf{B}}$ as 
    \begin{equation}
        \label{eq:JhatGradNoG}
            \begin{split}
                d\hat\phi^{\epsilon,i}_t=&d\hat B^i_t+\Biggl\{\frac{\beta(p-q)}{2}\mathrm{cot}\hat\phi^{\epsilon,i}_t+\beta(q-n+1)\mathrm{cot}(2\hat\phi^{\epsilon,i}_t)+\frac{\beta}{2}\sum_{j\ne i}\left[\mathrm{cot}(\hat\phi^{\epsilon,i}_t+\hat\phi^{\epsilon,j}_t)+\mathrm{cot}(\hat\phi^{\epsilon,i}_t-\hat\phi^{\epsilon,j}_t)\right]\Biggr\}dt\\
                &0\leq \hat\phi^{\epsilon,1}_t<\dots< \hat\phi^{\epsilon,n}_t\le\frac{\pi}{2}, \text{  a.s. } dt-\text{a.e.}
            \end{split}  
        \end{equation}
       By the same arguments as in the beginning of the proof, the normaly reflected SDE

       \begin{align}
        d\hat{\hat{\varphi}}_t & =  d\mathbf{\hat{ B}}_t-\nabla\mathcal{V}(\hat{\hat{\varphi}}_t)dt-\nu(\hat{\hat{\varphi}}_t)d\hat L_t \text{ for all  }t\geq0\label{eq:hatJgradnoGCepa}\\
        & \forall t\geq0,\hat{\hat{\varphi}}_t\in \Bar D \text{ a.s.}\nonumber\\
        &\hat{\hat{\varphi}}_0=\left(\arcsin\left(\sqrt{\lambda^1_0}\right),\dots,\arcsin\left(\sqrt{\lambda^n_0}\right)\right),\nonumber
       \end{align}
      admits a global solution and the term $\nu(\hat{\hat{\varphi}}_t)d\hat L_t $ is zero, so that $\hat{\hat{\varphi}}$ solves \eqref{eq:JhatGradNoG}. We can apply \cite[Theorem 3.1]{MR2792586}, the proof of which can  be applied to our case of SDE (\ref{eq:hatJgradnoGCepa}) to conclude that its solutions cannot have multiple collisions, which implies that there is no multiple collision  of $(\hat\phi^{1,\epsilon}_t,\dots,\hat\phi^{n,\epsilon}_t)_{t}$  under the probability $\mathbb{Q}$. There is thus no multiple collision of $(\hat\phi^{1,\epsilon}_t,\dots,\hat\phi^{n,\epsilon}_t)_{t\geq0}$ under the probability $\mathbb{P}$  :
    \begin{equation*}
          \mathbb{P}\{\exists t\ge0:\hat\phi^{i,\epsilon}_t=\hat\phi^{i+1,\epsilon}_t \text{ and } \hat\phi^{j,\epsilon}_t=\hat\phi^{j+1,\epsilon}_t \text{ for some } 0\leq i<j\leq n\}=0.
    \end{equation*}

    \end{proof}

\begin{lemma}
    \label{lemma:Jtilde}
    Let us assume $p\ge q$ and $q-n+1>0$.  The system of SDEs (\ref{eq:Jtilde}) with random initial condition $(\tilde\lambda^{1,\epsilon}_0,\dots,\tilde\lambda^{n,\epsilon}_0)$ such that $0\leq\tilde\lambda^{1,\epsilon}_0\leq\dots\leq\tilde\lambda^{n,\epsilon}_0\leq1 $ a.s. and independent from the Brownian motion $\mathbf{B}=(B^1_t,\dots,B^n_t)_t$ has a global pathwise unique strong solution $(\tilde{\lambda}^{1,\epsilon}_t,\dots,\tilde{\lambda}^{1,\epsilon}_t)_{t\geq0}$.
  
  Moreover, 
  
  \begin{eqnarray}
    \label{multiplecollisionbeta}
    \mathbb{P}\{\exists t>0:\tilde\lambda^{i,\epsilon}_t=\tilde\lambda^{i+1,\epsilon}_t \text{ and } \tilde\lambda^{j,\epsilon}_t=\tilde\lambda^{j+1,\epsilon}_t \text{ for  some } 2\leq i<j\leq n\}=0.
  \end{eqnarray}
\end{lemma}

    \begin{proof}

    Let us consider a  Brownian motion $\Tilde{\textbf{B}}=(\tilde B^1_t,\dots,\tilde B^n_t)_t$, and the system of SDEs

    \begin{equation*}
        \begin{split}
            d\tilde\lambda^{1,\epsilon}_t = \ & 2\sqrt{\tilde\lambda^{1,\epsilon}_t(1-\tilde\lambda^{1,\epsilon}_t)}d\tilde B^1_t+\beta\left[p-n+1-(p+q)\tilde\lambda^{1,\epsilon}_t\right]dt \\
  \forall i\in\{2,\dots, n \}\text{, }  d \tilde\lambda^{i,\epsilon}_t =&\ 2\sqrt{\tilde\lambda^{i,\epsilon}_t(1-\tilde\lambda^{i,\epsilon}_t)}d\tilde B^i_t\\
  &+\Biggl\{(1-2\tilde\lambda^{i,\epsilon}_t)\left[2-0\vee\frac{2\sqrt{2}}{\sqrt{\epsilon}}\left(\sqrt{1-\tilde\lambda^{i,\epsilon}_t}-\frac{\sqrt{\epsilon}}{2\sqrt{2}}\right)\wedge1 -0\vee\frac{2}{\sqrt{\epsilon}}\left(\sqrt{\tilde\lambda^{i,\epsilon}_t}-\frac{\sqrt{\epsilon}}{2}\right)\wedge1 \right]\\
  &+\beta\left[p-1-(p+q)\tilde\lambda^{i,\epsilon}_t+\sum_{j\geq2,j\neq i}\frac{\tilde\lambda^{i,\epsilon}_t(1-\tilde\lambda^{j,\epsilon}_t)+\tilde\lambda^{j,\epsilon}_t(1-\tilde\lambda^{i,\epsilon}_t)}{\tilde\lambda^{i,\epsilon}_t-\tilde\lambda^{j,\epsilon}_t}\right]\Biggl\}dt\\
  &0\leq \tilde\lambda^{1,\epsilon}_t\leq1 \text{ and } 0\leq \tilde\lambda^{2,\epsilon}_t<\dots< \tilde\lambda^{n,\epsilon}_t\leq1, \text{ a.s.,  } dt-a.e.
        \end{split}
  \end{equation*}
    Let us remark that $\tilde\lambda^1$ is a one dimensional Jacobi  process and admits a strong solution defined on $\R_+$ (see Lemma \ref{lemma:doumerc}), and that the coordinates $i\in\{2,\dots,n\}$ satisfy an autonomous system of SDEs for $n-1$ particles similar to equation (\ref{eq:Jhat}). The only differences are  the terms $\left[0\vee\frac{2\sqrt{2}}{\sqrt{\epsilon}}\left(\sqrt{\hat\lambda^{i,\epsilon}_t}-\frac{\sqrt{\epsilon}}{2\sqrt{2}}\right)\wedge1\right]$ and $\beta p$ in (\ref{eq:Jhat}) which become respectively $\left[0\vee\frac{2}{\sqrt{\epsilon}}\left(\sqrt{\tilde\lambda^{i,\epsilon}_t}-\frac{\sqrt{\epsilon}}{2}\right)\wedge1\right]$ and $\beta(p-1)$ here. We can still consider the  process $\left(\arcsin\left(\sqrt{\tilde\lambda^{2,\epsilon}_t}\right),\dots,\arcsin\left(\sqrt{\tilde\lambda^{n,\epsilon}_t}\right)\right)_t$, and apply the same method as in the proof of Lemma \ref{lemma:Jhat} (indeed, if $ q -n+1>0$, then $ q -(n-1)+1>0$) to prove the existence of a pathwise unique strong solution to this subsystem. As the equation for $\tilde\lambda^{1,\epsilon}$ also has a strong solution, the whole n-particles system considered here admits a global strong solution.
    
    Let us define for all $\epsilon>0$ and for  $t\geq0$ :   $\Theta(t)=(\theta_1(t),\dots,\theta_n(t))$ with
    \begin{eqnarray}
      \theta_1(t) &=& -\frac{\beta}{\sqrt{\tilde\lambda^{1,\epsilon}_t(1-\tilde\lambda^{1,\epsilon}_t)}}\times\frac{(1-\tilde\lambda^{1,\epsilon}_t)\wedge\frac{\epsilon}{2}}{\frac{\epsilon}{2}}\sum_{j\neq1}\frac{\tilde\lambda^{1,\epsilon}_t(1-\tilde\lambda^{j,\epsilon}_t)}{(\tilde\lambda^{j,\epsilon}_t-\tilde\lambda^{1,\epsilon}_t\wedge\epsilon)\vee\epsilon},\nonumber \\
      \forall i\in\{2,\dots, n \}\text{, } \theta_i(t) & = &\frac{\beta}{\sqrt{\tilde\lambda^{i,\epsilon}_t(1-\tilde\lambda^{i,\epsilon}_t)}}\times \frac{(1-\tilde\lambda^{i,\epsilon}_t)\wedge\frac{\epsilon}{2}}{\frac{\epsilon}{2}}\times\frac{\tilde\lambda^{i,\epsilon}_t(1-\tilde\lambda^{1,\epsilon}_t)}{(\tilde\lambda^{i,\epsilon}_t-\tilde\lambda^{1,\epsilon}_t)\vee\epsilon}, \nonumber
    \end{eqnarray}
    and for all $t\geq0$
    \begin{equation}
        Z(t)=\exp\left\{\int_0^t\Theta(u)\cdot d\Tilde{\textbf{B}}_u-\frac{1}{2}\int_0^t||\Theta(u)||^2du\right\}. \nonumber
    \end{equation}

    We have
    \begin{eqnarray*}
      \theta_1^2(t) & = & \frac{\beta^2}{\tilde\lambda^{1,\epsilon}_t(1-\tilde\lambda^{1,\epsilon}_t)}\left(\frac{(1-\tilde\lambda^{1,\epsilon}_t)\wedge\frac{\epsilon}{2}}{\frac{\epsilon}{2}}\sum_{j\neq1}\frac{\tilde\lambda^{1,\epsilon}_t(1-\tilde\lambda^{j,\epsilon}_t)}{(\tilde\lambda^{j,\epsilon}_t-\tilde\lambda^{1,\epsilon}_t\wedge\epsilon)\vee\epsilon}\right)^2\\
      &\leq&\beta^2\tilde\lambda^{1,\epsilon}_t\frac{(1-\tilde\lambda^{1,\epsilon}_t)\wedge\frac{\epsilon}{2}}{\frac{\epsilon^2}{4}}\left(\sum_{j\neq1}\frac{1-\tilde\lambda^{j,\epsilon}_t}{(\tilde\lambda^{j,\epsilon}_t-\tilde\lambda^{1,\epsilon}_t\wedge\epsilon)\vee\epsilon}\right)^2\\
      & \leq & \frac{2\beta^2}{\epsilon}\left(\sum_{j>1}\frac{1}{\epsilon}\right)^2 \\
      &\leq&  \frac{2(n-1)^2\beta^2}{\epsilon^3}.
    \end{eqnarray*}

    For all $2<i\leq n$,
    \begin{eqnarray*}
      \theta_i^2(t) & = & \frac{\beta^2}{\tilde\lambda^{i,\epsilon}_t(1-\tilde\lambda^{i,\epsilon}_t)}\times \frac{\left((1-\tilde\lambda^{i,\epsilon}_t)\wedge\frac{\epsilon}{2}\right)^2}{\frac{\epsilon^2}{4}}\times\frac{\left(\tilde\lambda^{i,\epsilon}_t(1-\tilde\lambda^{1,\epsilon}_t)\right)^2}{\left((\tilde\lambda^{i,\epsilon}_t-\tilde\lambda^{1,\epsilon}_t)\vee\epsilon\right)^2}\\
      & \leq & \frac{2\beta^2}{\epsilon}\times\frac{\tilde\lambda^{i,\epsilon}_t\left(1-\tilde\lambda^{1,\epsilon}_t\right)^2}{\epsilon^2}\\
      &\leq &\frac{2\beta^2}{\epsilon^3}.
    \end{eqnarray*}

    We thus have
    \begin{eqnarray}
      \mathbb{E}\left[\exp\left\{\frac{1}{2}\int_0^t||\Theta(u)||^2du\right\}\right]& <&\infty \text{ for all }t\geq0.\nonumber
    \end{eqnarray} 
    Then, according to Novikov's criterion (see for instance \cite[Proposition 5.12 p.198]{MR1121940}), $Z$ is a $\mathbb{P}$-martingale, and $\mathbb{E}[Z(t)]=1$ for all $t\geq0$. Consequently, defining  for all $t\ge0$, $\tilde{\mathcal{F}}_t=\sigma\left((\tilde{\mathbf{B}}_t)_{s\le t},(\tilde\lambda_0^{1,\epsilon},\dots,\tilde\lambda_0^{n,\epsilon})\right)$ and $\mathbb{Q}$ such that
    $$\frac{d\mathbb{Q}}{d\mathbb{P}}_{|\tilde{\mathcal{F}}_t}=Z(t),$$ and
    \begin{eqnarray}
     \tilde{ \tilde {B}}^1_t&=&\tilde B^1_t-\int_0^t\theta_1(s)ds \nonumber \\
      & = & \tilde B^1_t+\int_0^t\frac{\beta}{\sqrt{\tilde\lambda^{1,\epsilon}_s(1-\tilde\lambda^{1,\epsilon}_s)}}\times\frac{(1-\tilde\lambda^{1,\epsilon}_s)\wedge\frac{\epsilon}{2}}{\frac{\epsilon}{2}}\sum_{j\neq1}\frac{\tilde\lambda^{1,\epsilon}_s(1-\tilde\lambda^{j,\epsilon}_s)}{(\tilde\lambda^{j,\epsilon}_s-\tilde\lambda^{1,\epsilon}_s\wedge\epsilon)\vee\epsilon}ds\nonumber\\
        \text{ for all } i\in\{2,\dots,n\},\text{ } \tilde{\tilde{ B}}^i_t  &=&\tilde B^i_t-\int_0^t\theta_i(s)ds \nonumber \\
        & = & \tilde B^i_t-\int_0^t\frac{\beta}{\sqrt{\tilde\lambda^{i,\epsilon}_s(1-\tilde\lambda^{i,\epsilon}_s)}}\times \frac{(1-\tilde\lambda^{i,\epsilon}_s)\wedge\frac{\epsilon}{2}}{\frac{\epsilon}{2}}\times\frac{\tilde\lambda^{i,\epsilon}_s(1-\tilde\lambda^{1,\epsilon}_s)}{(\tilde\lambda^{i,\epsilon}_s-\tilde\lambda^{1,\epsilon}_s)\vee\epsilon} ds, \text{ }0\leq t,\nonumber
    \end{eqnarray}
    $\Tilde{\Tilde{\textbf{B}}}=(\tilde{\tilde{ B}}^1_t,\dots,\tilde{\tilde{ B}}^n_t)_t$ is a $\mathbb{Q}$- Brownian motion
    according to the Girsanov theorem (see for instance \cite[Proposition 5.4 p.194]{MR1121940}).

  Consequently, (\ref{eq:Jtilde}) has a global weak solution.
  
  We now have to prove the pathwise uniqueness of the solutions to (\ref{eq:Jtilde}).
  The differences with Lemma \ref{lemma:pathwise_uniqueness} are the terms \[(1-2\tilde\lambda^{i,\epsilon}_t)\left[2-0\vee\frac{2\sqrt{2}}{\sqrt{\epsilon}}\left(\sqrt{1-\tilde\lambda^{i,\epsilon}_t}-\frac{\sqrt{\epsilon}}{2\sqrt{2}}\right)\wedge1 -0\vee\frac{2}{\sqrt{\epsilon}}\left(\sqrt{\tilde\lambda^{i,\epsilon}_t}-\frac{\sqrt{\epsilon}}{2}\right)\wedge1 \right],\] and the interaction terms between the first particle and the others. 
  
  Let  $Z=(z^1_t,\dots,z^n_t)_t$ and $\tilde Z=(\tilde z^1_t,\dots ,\tilde z^n_t)_t$  be two global solutions to (\ref{eq:Jtilde}) with $Z_0=\tilde Z_0$ independent from the same driving Brownian motion $\textbf{B}=(B^1_t,\dots,B^n_t)_t$.

    The local time of $z^i-\tilde z^i$ at $0$ is zero (\cite[Lemma 3.3 p.389]{MR1725357}). Applying the Tanaka formula to the process $z^i-\tilde z^i$ and summing over $i$,

    \begin{eqnarray}
      \sum_{i=1}^n|z^i_{t}-\tilde z^i_{t}| 
      & = &2\sum_{i=1}^n\int_0^{t}\mathrm{sgn}(z^i_s-\tilde z^i_s)\left(\sqrt{z^i_s(1-z^i_s)}-\sqrt{\tilde z^i_s(1-\tilde z^i_s)}\right)dB^i_s\nonumber\\
      &&+ 2\beta\int_0^{t}\sum_{i=2}^n\text{sgn}(z^i_s-\tilde z^i_s)\sum_{j\geq2, j\neq i}\left(\frac{z^i_s(1-z^j_s)}{z^i_s-z^j_s}-\frac{\tilde z^i_s(1-\tilde z^i_s)}{\tilde z^i_s-\tilde z^j_s}\right)ds\label{pu1}\\
      && -\beta(p+q)\int_0^{t}\sum_{i=1}^n|z^i_s-\tilde z^i_s|ds\label{pu2}\\
      &&+2\beta\int_0^{t}\sum_{i=2}^n\Bigg\{\text{sgn}(z^i_s-\tilde z^i_s)\left(\frac{(1-z^i_s)\wedge\frac{\epsilon}{2}}{\frac{\epsilon}{2}}.\frac{z^i_s(1-z^1_s)}{(z^i_s-z^1_s)\vee\epsilon}-\frac{(1- \tilde z^i_s)\wedge\frac{\epsilon}{2}}{\frac{\epsilon}{2}}.\frac{\tilde z^i_s(1-\tilde z^1_s)}{(\tilde z^i_s-\tilde z^1_s)\vee\epsilon}\right)\nonumber\\
      &&-\text{sgn}(z^1_s-\tilde z^1_s)\left(\frac{(1-z^1_s)\wedge\frac{\epsilon}{2}}{\frac{\epsilon}{2}}\frac{z^1_s(1-z^i_s)}{(z^i_s-z^1_s\wedge\epsilon)\vee\epsilon}-\frac{(1-\tilde z^1_s)\wedge\frac{\epsilon}{2}}{\frac{\epsilon}{2}}\frac{\tilde z^1_s(1-\tilde z^i_s)}{(\tilde z^i_s-\tilde z^1_s\wedge\epsilon)\vee\epsilon}\right)\Bigg\}ds\label{pu3}\\
      &&-\int_0^{t}\sum_{i=2}^n\text{sgn}(z^i_s-\tilde z^i_s)\Biggl\{(1-2z^i_s)\left[2-0\vee\frac{2\sqrt{2}}{\sqrt{\epsilon}}\left(\sqrt{1-z^i_s}-\frac{\sqrt{\epsilon}}{2\sqrt{2}}\right)\wedge1 -0\vee\frac{2}{\sqrt{\epsilon}}\left(\sqrt{z^i_s}-\frac{\sqrt{\epsilon}}{2}\right)\wedge1 \right]\nonumber\\
      &&-(1-2\tilde z^i_s)\left[2-0\vee\frac{2\sqrt{2}}{\sqrt{\epsilon}}\left(\sqrt{1-\tilde z^i_s}-\frac{\sqrt{\epsilon}}{2\sqrt{2}}\right)\wedge1 -0\vee\frac{2}{\sqrt{\epsilon}}\left(\sqrt{\tilde z^i_s}-\frac{\sqrt{\epsilon}}{2}\right)\wedge1 \right]\Biggr\}ds.\label{pu4}
    \end{eqnarray}

    As in the proof of Lemma \ref{lemma:pathwise_uniqueness}, the expectation of the stochastic integrals is zero and the terms (\ref{pu1}) are not positive (see  inequality (\ref{sgncomputation})). To deal with the expectation of (\ref{pu3}), one  remarks that the function \[(x,y)\mapsto\frac{(1-x)\wedge\frac{\epsilon}{2}}{\frac{\epsilon}{2}}.\frac{x(1-y)}{(x-y)\vee\epsilon}\] is Lipschitz on $[0,1]^2$.

     As for the term (\ref{pu4}), the function $f:z\mapsto(1-2z)\left[2-0\vee\frac{2\sqrt{2}}{\sqrt{\epsilon}}\left(\sqrt{1-z}-\frac{\sqrt{\epsilon}}{2\sqrt{2}}\right)\wedge1 -0\vee\frac{2}{\sqrt{\epsilon}}\left(\sqrt{z}-\frac{\sqrt{\epsilon}}{2}\right)\wedge1 \right]$ defined on $[0,1]$ is continuous and can be rewritten 
     \begin{equation*}
       f(z) = \left\{\begin{aligned}
         &1-2z \text{ if }z\in\left[0,\frac{\epsilon}{4}\right)\cup\left(1-\frac{\epsilon}{8},1\right],\\
         &(1-2z)\left(\frac{2\sqrt{z}}{\sqrt{\epsilon}}-1\right) \text{ if }z\in\left[\frac{\epsilon}{4},\epsilon\right],\\
         &(1-2z)\left(\frac{2\sqrt{2(1-z)}}{\sqrt{\epsilon}}-1\right) \text{ if }z\in\left[1-\frac{\epsilon}{2},1-\frac{\epsilon}{8}\right],\\
         &0 \text{ if }z\in \left[\epsilon,1-\frac{\epsilon}{2}\right],
       \end{aligned}\right.
     \end{equation*}
     and is thus Lipschitz on $[0,1]$.
     
     Then   there exists a constant $K\geq0$  such that for all $t\geq0$
     
      \begin{eqnarray}
      \sum_{i=1}^n\mathbb{E}|z^i_{t}-\tilde z^i_{t}| & \leq & K\mathbb{E}\left[\int_0^{t}\sum_{i=1}^n|z^i_{s}-\tilde z^i_{s}|ds\right]\nonumber\\
      & \leq & K\int_0^{t}\sum_{i=1}^n\mathbb{E}|z^i_{s}-\tilde z^i_{s}|ds.\nonumber
    \end{eqnarray}
     
   The Grönwall Lemma allows to conclude that for all $t\geq0$
   \begin{equation}
      \sum_{i=1}^n\mathbb{E}|z^i_{t}-\tilde z^i_{t}| =0,\nonumber 
  \end{equation}
  which concludes the proof on the existence and pathwise uniqueness.
  
   Let us now prove (\ref{multiplecollisionbeta}). To do so, we apply to the process 
  $\left(\arcsin\left(\sqrt{\tilde \lambda^{2,\epsilon}_t}\right),\dots,\arcsin\left(\sqrt{\tilde \lambda^{n,\epsilon}_t}\right)\right)_t$ the method used to prove (\ref{eq:Jhatnonmultiplecollision}). Indeed, this process solves, as explained in the beginning of the proof, a system of SDEs similar to (\ref{eq:Jhat}) for $n-1$ particles.

    \end{proof}
\begin{lemma}
    \label{lemma:Jcheck}
    Let us assume $p\ge q$ and $q-n+1>0$.  The system of SDEs (\ref{eq:Jcheck}) with random initial condition $(\check\lambda^{1,\epsilon}_0,\dots,\check\lambda^{n,\epsilon}_0)$ such that $0\leq\check\lambda^{1,\epsilon}_0\leq\dots\leq\check\lambda^{n,\epsilon}_0\leq1 $ a.s. and independent from the Brownian motion $\mathbf{B}=(B^1_t,\dots,B^n_t)_t$ has a global pathwise unique strong solution $(\check{\lambda}^{1,\epsilon}_t,\dots,\check{\lambda}^{n,\epsilon}_t)_{t\geq0}$.
  
  Moreover, 
  
  \begin{eqnarray}
    \mathbb{P}\{\exists t>0:\check\lambda^{i,\epsilon}_t=\check\lambda^{i+1,\epsilon}_t \text{ and } \check\lambda^{j,\epsilon}_t=\check\lambda^{j+1,\epsilon}_t \text{ for some } 0\leq i<j\leq n-2\}=0.\nonumber
  \end{eqnarray}
\end{lemma}
\begin{proof}
    The proof of this Lemma follows the same steps as the proof of Lemma \ref{lemma:Jtilde} interchanging the role of the particles indexed by $1$ and $n$.
\end{proof}
\begin{lemma}
    \label{lemma:Jbar}
    Let us assume $p\ge q$ and $q-n+1>0$.  The system of SDEs (\ref{eq:Jbar}) with random initial condition $(\bar\lambda^{1,\epsilon}_0,\dots,\bar\lambda^{n,\epsilon}_0)$ such that $0\leq\bar\lambda^{1,\epsilon}_0\leq\dots\leq\bar\lambda^{n,\epsilon}_0\leq1 $ a.s. and independent from the Brownian motion $\mathbf{B}=(B^1_t,\dots,B^n_t)_t$ has a global pathwise unique strong solution $(\bar{\lambda}^{1,\epsilon}_t,\dots,\bar{\lambda}^{n,\epsilon}_t)_{t\geq0}$.
  
  Moreover, 
  
  \begin{eqnarray}
    \mathbb{P}\{\exists t>0:\bar\lambda^{i,\epsilon}_t=\bar\lambda^{i+1,\epsilon}_t \text{ and } \bar\lambda^{j,\epsilon}_t=\bar\lambda^{j+1,\epsilon}_t \text{ for some } 2\leq i<j\leq n-2\}=0.\nonumber
  \end{eqnarray}
\end{lemma}

\begin{proof}
    The proof of this Lemma follows the same steps as in Lemma \ref{lemma:Jtilde}, but with both $\bar\lambda^{1,\epsilon}$ and $\bar\lambda^{n,\epsilon}$ playing a peculiar role like $\tilde\lambda^{1,\epsilon}$ in (\ref{eq:Jtilde}) and $\check\lambda^{n,\epsilon}$ in (\ref{eq:Jcheck}).
\end{proof}

  \section{Appendix}
  This Lemma deals with the existence and uniqueness to the real Jacobi process and was proved in \cite{doumerc2005matrices}.
  \begin{lemma}[Real Jacobi Processes]
  \label{lemma:doumerc}
  The real Jacobi process of parameters $d,d'\ge0$ is the unique strong solution to
  \begin{equation*}
      dJ_t = 2\sqrt{J_t(1-J_t)}dB_t+(d-(d+d')J_t)dt
  \end{equation*}
  where $B$ is a real Brownian motion. The process $J$ remains in $(0,1)$ a.s. when $d\wedge d'\ge2$, hits $0$ a.s. if $0\leq d<2$ and hits $1$ a.s. if $0\le d'<2$.
  \end{lemma}

  The next lemma deals with the existence and uniqueness to the CIR SDE and with the probability for the solution to hit zero. It is proved for instance in \cite[Theorem 6.2.2 and Proposition 6.2.3]{MR2362458}.
  \begin{lemma}
  \label{lamberton}
Let $a\geq0,b,\sigma\in\mathbb{R}$.  Suppose that $W$ is a standard Brownian motion defined on $\mathbb{R}_+$. For any real number $x\geq0$, there is a unique continuous, adapted process $X$, taking values in $\mathbb{R}_+$, satisfying $X_0=x$ and $$dX_t=(a-bX_t)dt+\sigma\sqrt{X_t}dW_t \text{ on } [0,\infty).$$
  Moreover, if we denote by $X^x$ the solution to this SDE starting at $x$ and by $\tau_0^x=\inf\{t\geq0 : X_t^x=0\}$,
  \begin{enumerate}
      \item If $a\geq\sigma^2/2$,  we have $\mathbb{P}(\tau^x_0=\infty)=1$, for all $x>0$.
      \item If $0\leq a <\sigma^2/2$ and $b\geq0$, we have  $\mathbb{P}(\tau^x_0<\infty)=1$, for all $x>0$.
      \item If $0\leq a <\sigma^2/2$ and $b<0$, we have  $0<\mathbb{P}(\tau^x_0<\infty)<1$, for all $x>0$.
  \end{enumerate}
  \end{lemma}
  
      The following result is the  Ikeda-Watanabe Theorem, which allows to compare two Itô processes if their starting points and their drift coefficients are comparable, and if their diffusion coefficients  are regular enough. It is proved for instance in \cite[Theorem V.43.1 p.269]{MR1780932}.
 \begin{theorem}(Ikeda-Watanabe)
 \label{Ikeda}
 Suppose that, for $i=1,2$, 
 \begin{equation}
     X^i_t = X^i_0+\int_0^t\sigma(X^i_s)dB_s+\int_0^t\beta_s^ids,
 \end{equation}
 and that there exist $b:\mathbb{R}\mapsto\mathbb{R}$, such that
 $$\beta^1_s\geq b(X^1_s)\text{, }b(X^2_s)\geq \beta^2_s.$$
 Suppose also that
 \begin{enumerate}
     \item $\sigma$ is measurable and there exists an increasing function $\rho:\mathbb{R}_+\mapsto\mathbb{R}_+$ such that $$\int_{0^+}\rho(u)^{-1}du=\infty,$$ and for all $x,y\in\mathbb{R},$
     $$(\sigma(x)-\sigma(y))^2\leq\rho(|x-y|);$$
     \item $X^1_0\geq X^2_0$ a.s.;
     \item $b$ is Lipschitz.
 \end{enumerate}
 Then $X^1_t\geq X^2_t$ for all $t$ a.s.
 \end{theorem}

\nocite{*}
\bibliographystyle{amsalpha}
\bibliography{jacobi.bib}
\end{document}